 \documentclass[draft]{article}

\usepackage{amsmath,amsfonts,amsthm,amssymb,amscd,color, ulem}
\renewcommand{\Re}{\mathop{\rm Re}\nolimits}
\renewcommand{\Im}{\mathop{\rm Im}\nolimits}

\theoremstyle{plain} \newtheorem{theorem}{Theorem}[section]
\newtheorem{lemma}[theorem]{Lemma}
\newtheorem{proposition}[theorem]{Proposition}
 \theoremstyle{definition}
\newtheorem{definition}[theorem]{Definition} \theoremstyle{remark}
\newtheorem{remark}[theorem]{Remark}

\newcommand{\R}{{\mathbb R}}

\newcommand{\Z}{{\mathbb Z}}

\newcommand{\N}{{\mathbb N}}

\def\im{{\rm i}}

\newcommand{\C}{\mathbb{C}} 
\newcommand{\T}{\mathbb{T}}

\def\({\left(}
\def\){\right)}
\def\<{\left\langle}
\def\>{\right\rangle}

\newcommand{\stz}{\mathrm{Stz}}

\numberwithin{equation}{section}

\begin{document}

\title{Stabilization of small solutions of discrete NLS with potential having two eigenvalues}

\author {Masaya Maeda}

\maketitle

\begin{abstract}
We study the long time behavior of small (in $l^2$) solutions of discrete nonlinear Schr\"odinger equations with potential.
In particular, we are interested in the case that the corresponding discrete Schr\"odinger operator has exactly two eigenvalues.
We show that under the nondegeneracy condition of Fermi Golden Rule, all small solutions decompose into a nonlinear bound state and dispersive wave.
We further show the instability of excited states and generalized equipartition property.
\end{abstract}

\section{Introduction}
In this paper, we consider the following discrete nonlinear Schr\"odinger equation (DNLS) on $\Z$:
\begin{align}\label{1}
\im \partial _ t u = H u +\beta(|u|^2 ) u,\quad u:\R\times \Z\to \C,
\end{align}
where, $H:=-\Delta + V$ and $\Delta$ is the discrete Laplacian:
\begin{align*}
(\Delta u)(n):=u(n+1)-2u(n) + u(n-1).
\end{align*}
Moreover, we set
$(Vu)(n):=V(n)u(n)$ with $\sum_{n\in\Z} (1+|n|)|V(n)|<  \infty$ (in particular $V(n)\to 0$ as $|n|\to  \infty$) and 
\begin{align}\label{betapoly}
\beta(s)=s^3+ \sum_{j=4}^M \lambda_j s^j,\quad (M\in \N,\ \lambda_j\in \R)  .
\end{align}
In the following, we always assume that $0,4$ are not resonances nor eigenvalues.

\begin{remark}
We need to assume $\beta(s)=O(s^3)$ for technical reason related to the slow decay of the linear solution.
The sign of the nonlinearity is irrelevant to our discussion because we will consider only small solutions.
\end{remark}

\begin{remark}\label{rem:essspec}
We have $\sigma(-\Delta)=\sigma_{\mathrm{ess}}(-\Delta)=[0,4]$, where $\sigma(-\Delta)$ (resp.\ $\sigma_{\mathrm{ess}}(-\Delta)$) is the set of spectrum (essential spectrum) of $-\Delta$.
Therefore, we also have $\sigma_{\mathrm{ess}}(H)=[0,4]$.
\end{remark}

The (continuous) nonlinear Schr\"odinger equations are universal model which describe wave propagation in weakly nonlinear media with dispersion.
Similarly, DNLS type equations appear in various regions in physics such as coupled optical waveguides \cite{Eisenberg02JOSAB, Peschel02JOSAB}, photonic lattice \cite{Efremidis02PRE, Sukhorukov03IEEEJQE}, Bose-Einstein condensation \cite{Cataliotti01Science} and nonlinear Su-Schrieffer-Heeger model describing topological insulator \cite{Hadad16PRB}.
We further refer \cite{FG08PR, FW98PR} for the discussion of the role of the linear potential in DNLS.

We are interested in the long time behavior of general small solutions of DNLS \eqref{1}.
By small solutions, we mean solutions of \eqref{1} with initial data $u(0,\cdot)=u_0\in l^2$ with $\|u_0\|_{l^2}^2:=\sum_{n\in \Z}|u_0(n)|^2$ sufficiently small.
Notice that by the potential $V$, the discrete Schr\"odinger operator $H=-\Delta+V$ may have eigenvalues.
In this case one can show that there exist nonlinear bound states associated to the eigenvalues of $H$. Here, a nonlinear bound state is a solution of DNLS \eqref{1} with the form $e^{-\im \omega t}\phi_\omega(n)$ (see Proposition \ref{prop:1}. Further, for other types of nonlinear bound states see \cite{BP10N}).

When $H$ has no eigenvalues, it is known that all small (in $l^2$) solutions scatter.
By scattering, we mean that there exists $\eta_+\in l^2$ s.t.\ the solution converges (in $l^2$) to the free solution $e^{\im t \Delta}\eta_+$ as $t\to \infty$.
For the case $V\equiv 0$ this was shown by Stefanov--Kevrekidis \cite{SK05N}.
For the case $V\neq 0$, it follows from the dispersive estimate of $H$ proved by Pelinovsky--Stefanov \cite{PS08JMP} (see also \cite{KKK06AA} and for lower power nonlinearity case, see \cite{MP10AA}).
However, we do not know an example s.t.\ $V\neq 0$ and $-\Delta+V$ has no eigenvalues (see section 4 and appendix of \cite{KKK06AA}).

When $H$ has one eigenvalue, it is known that all small solutions decouple into a nonlinear bound state and dispersive wave.
This means that after subtracting suitable nonlinear bound state from the solution, the remainder scatters.
Therefore, the solution $u(t)$ can be expressed as
\begin{align}\label{1evcase}
u(t)=\phi(z(t))+e^{\im t \Delta}\eta_+ + \mathrm{error}(t),\quad \|\mathrm{error}(t)\|_{l^2}\to 0,
\end{align}
where the nonlinear bound state $\phi$ is parametrized by $z\in \C$ (see Proposition \ref{prop:1}).
This was shown by Cuccagna--Tarulli \cite{CT09SIAM} and Kevrekidis--Pelinovsky--Stefanov \cite{KPS09SIAM} independently (see also \cite{MP12DCDS} for lower power nonlinearity case).
We remark that similar results also hold for the continuous nonlinear Schr\"odinger equations (NLS) on $\R^d$ when the Schr\"odinger operator has exactly one eigenvalue (see, \cite{GNT04, Miz07JMKU, Miz08JMKU, PW97JDE, SW90CMP}).
Notice that by the spectral decomposition, the long time behavior given in \eqref{1evcase} is similar to the long time behavior of the linear discrete Schr\"odinger equation $\im u_t = Hu$.
This is quite natural to expect because if the amplitude of the solution is small, then the nonlinear term will be much smaller than the linear term.

We now come to the case that $H$ has two eigenvalues.
We set 
\begin{align}\label{omegan}
\sigma_d(H)=\{e_1<e_2\}\quad
\text{and}\quad \omega_n:=e_1+n(e_2-e_1),
\end{align}
where $\sigma_d(H)$ is the set of eigenvalues (discrete spectrum) of $H$.
We further set $\phi_j$ to be the real valued normalized eigenfunctions of $H$ associated to $e_j$.
By the author  \cite{MDNLS1}, it was shown that if we assume
\begin{align}\label{2}
\omega_n\notin [0,4]=\sigma_{\mathrm{ess}}(H),\quad \forall n\in\Z,
\end{align}
then there exists a 2-parameter family of quasi-periodic solutions $\psi(z_1,z_2)= z_1\phi_1+z_2\phi_2 +o(|z|)$ and all small solutions of DNLS \eqref{1} decouple into a quasi-periodic solution and dispersive wave.
Notice that this is also similar to the behavior of linear discrete Schr\"odinger equation because general solutions can be expressed as
\begin{align}\label{lineardynamics}
u(t)=z_1 e^{\im e_1 t}\phi_1+z_2 e^{\im e_2 t}\phi_2 + e^{-\im t H}P_c u(0),
\end{align}
where $z_j\in \C$ are constants and $P_c$ is the projection to the continuous spectrum of $H$.
Further,  by linear scattering, there exists $\eta_+\in l^2$ s.t.
\begin{align*}
e^{\im t H}P_c u(0)=e^{\im t \Delta}\eta_+ + \mathrm{error}(t),\quad \|\mathrm{error}(t)\|_{l^2}\to 0\text{ as }t\to \infty.
\end{align*} 

In this paper, we assume that $\omega_n\neq 0,4$ for all $n$ and there exists $N_0\in \Z$ s.t.
\begin{align}\label{2.1}
\omega_{N_0}\in (0,4).
\end{align}
\begin{remark}\label{rem:0}
If $e_1<0<4<e_2$, we have \eqref{2}.
Therefore, without loss of generality, we can assume $e_1<e_2<0$ and $\omega_{N_0-1}<0<\omega_{N_0}<4$ for some $N_0\geq 2$.
Notice that the case $4<e_1<e_2$ can be reduced to the previous case by the so-called staggering transform $\mathcal T u(n):=(-1)^{n}u(n)$.
By this transformation, the nonlinear term will change its sign but since we are only considering small solutions, it will make no change in the argument.
\end{remark}

We show that under the assumption \eqref{2.1} and the Fermi Golden Rule assumption (which we will explain below), all small (in $l^2$) solutions decouple into a nonlinear bound state and dispersive wave (Theorem \ref{thm:1}).
Thus, the solution $u(t)$ can be expressed as
\begin{align}\label{2evcase}
u(t)=\phi_j(z(t))+e^{\im t \Delta}\eta_+ + \mathrm{error}(t),\quad \|\mathrm{error}(t)\|_{l^2}\to 0,
\end{align}
where $\phi_j(z)=z\phi_j+o(|z|)$ is the nonlinear bound state (given in Proposition \ref{prop:1}) and $j$ will be $1$ or $2$ depending on the solution.
At first glance, one may think the result is similar to the one eigenvalue case because \eqref{1evcase} and \eqref{2evcase} looks similar.
However, comparing \eqref{2evcase} with the dynamics of linear discrete Schr\"odinger equation, there is a large difference because the solution of linear equation satisfies \eqref{lineardynamics}.
Notice that in \eqref{lineardynamics}, the solution has two bound states but in \eqref{2evcase}, the solution has only one bound state.
As a result, we see that there exists no quasi-periodic solution.
Therefore, combined with \cite{MDNLS1}, we see that the long time behavior of small solutions (in particular the existence of quasi-periodic solutions) heavily depends on the position of eigenvalues which generically satisfies \eqref{2} or \eqref{2.1}.

We now explain the role of two eigenvalues and the meaning of $\omega_n$.
For simplicity of explanation, we set the nonlinearity to be $|u|^2u$.
First, notice that by the gauge invariance of the nonlinearity, if we substitute $u=e^{-\im e_j t}\phi_j$ in $|u|^2u$, we get $e^{-\im e_j t}\phi_\omega^3$.
So, the nonlinearity do not change a single frequency.
However, if we substitute $u=e^{-\im e_1 t}\phi_1 + e^{-\im e_2 t}\phi_2$ in $|u|^2u$, we have
\begin{align*}
|u(t)|^2u(t)=e^{-\im e_1 t}(\phi_1^3+\phi_1\phi_2^2)+e^{-\im e_2 t}(\phi_2^3+\phi_1^2\phi_2)+e^{-\im (2 e_1-e_2)t}\phi_1^2\phi_2+e^{-\im(2 e_2-e_1)t}\phi_2^2\phi_1.
\end{align*}
Therefore, we see that new frequencies $\omega_{-1}=2e_1-e_2$ and $\omega_2=2e_2-e_1$ appear (note that $\omega_0=e_1$ and $\omega_1=e_2$, see \eqref{omegan}).
Similarly, the new frequencies will create more frequencies, and we will have that all frequencies $\omega_n$ $n\in \Z$ will be created by the nonlinearity.
We now see that the conditions \eqref{2} and \eqref{2.1} are about the resonance between these frequencies with the continuous spectrum of $H$ (recall Remark \ref{rem:essspec}).
In \cite{MDNLS1}, we have shown that if there is no resonance (which is the case of \eqref{2}), then there exists a family of quasi-periodic solutions (or in other words, the solution behaves similar to linear equation), and if there is a resonance, we will show in this paper, there exists no quasi-periodic solution (or the solution behaves differently compared to linear equation).
We refer \cite{FG08PR, FW98PR} for related discussion.

Recall that the essential spectrum of the continuous Schr\"odinger operator $H_c=-\sum_{j=1}^d \partial_{x_j}^2 + V$ is $[0,\infty)$.
Thus, the assumption \eqref{2} with $[0,4]$ replaced by $[0,\infty)$ can never be satisfied.
Therefore, one can expect that for the continuous NLS, all small solutions decouple into a nonlinear bound state and dispersive wave (and in particular no small quasi-periodic solution exists).
Indeed, for NLS on $\R^3$ this was shown by Soffer--Weinstein \cite{SW04RMP} and Tsai--Yau \cite{TY02ATMP} for the two eigenvalue cases with $N_0=2$ and Cuccagna--Maeda \cite{CuMaAPDE} for the general cases.
Therefore, our result in this paper is similar to the continuous NLS (for related results for nonlinear Klein-Gordon and Dirac equations, see \cite{CMP16NA} and \cite{CPS17AIHPAN, CT16JMAA,PS12JMP}).
For experimental realization, see \cite{MLS05PRL}.

When $H$ has more than $3$ eigenvalues, the situation becomes complicated.
This is because if a pair of eigenvalues $\{e_{m_1},e_{m_2}\}$ satisfies \eqref{2}, then one can construct a family of quasi-periodic solutions associated to the eigenfunctions of $\{e_{m_1}, e_{m_2}\}$.
On the other hand, if $\{e_{m_1},e_{m_2}\}$ satisfies \eqref{2.1}, then from our result, it is natural to think there will be no such quasi-periodic solution.
Further, we conjecture there will be no quasi-periodic solution with three modes because $\{\omega_{n,m}\}_{n,m\in\Z}$ is generically dense in $\R$, where $\omega_{n,m}=e_1+n(e_2-e_1)+m(e_3-e_1)$.
However, this will be a future work.

We introduce some notations to state our result precisely.
\begin{itemize}
\item
We often write $a\lesssim b$ by meaning that there exists a constant $C$ s.t.\  $a\leq Cb$.
If we have $a\lesssim b$ and $b\lesssim a$, we write $a\sim b$.
\item
For $p\geq 1$, $\sigma\in\R$, we set
$l^{p,\sigma}(\Z):=\left\{ u=\{u(n)\}_{n\in\Z}\ |\ \|u\|_{l^{p,\sigma}}^p:=\sum_{n\in\Z}\<n\>^{p\sigma}|u(n)|^p<\infty\right\},$ where $\<n\>:=(1+n^2)^{1/2}$.
Further, $l^p(\Z):=l^{p,0}(\Z)$ and 
we define the (real) inner-product of $l^2(\Z)$ by
$
\<u,v\>:=\Re\sum_{n\in \Z}u(n)\overline{v(n)}.
$
\item
For $a\in \R$, we set
$l^a_e(\Z):=\{u=\{u(n)\}_{n\in\Z}\ |\ \|u\|_{l^a_e}^2:=\sum_{n\in\Z} e^{2a|n|}|u(n)|^2<\infty\}.
$

\item
For a Banach space $X$ equipped with the norm $\|\cdot\|_X$,
we set
$
B_X(\delta):=\{u\in X\ |\ \|u\|_{X}<\delta\}.
$
\item
For Banach spaces $X,Y$, we set $\mathcal L(X;Y)$ to be the Banach space of all bounded operators from $X$ to $Y$, and $\mathcal L(X):=\mathcal L(X;X)$.
Further, we set $\mathcal L^n(X;Y)$ inductively by $\mathcal L^n(X;Y):=\mathcal L(X;\mathcal L^{n-1}(X;Y))$ and $\mathcal L^0(X;Y):=Y$.
\item
We set $C^\omega(B_X(\delta);Y)$ to be all real analytic functions from $B_X(\delta)$ to $Y$.
By real analytic functions, we mean that $f:B_X(\delta)\to Y$ can be written as $f(x)=\sum_{n\geq 0}a_n x^n$ with $\sum_{n\geq 0}\|a_n\|_{\mathcal L^n(X;Y)}r^n<\infty$ for all $r<\delta$, where $a_n\in \mathcal L^n(X;Y)$ and $a_nx^n:=a_n(x, x, \cdots, x)$.
\item For $\omega\in (0,4)$, we define $R_+(\omega)$ by $\lim_{\delta\downarrow 0}(H-\omega-\im \delta)^{-1}$, where the limit is taken in the space $\mathcal L(l^{2,\sigma}(\Z),l^{2,-\sigma}(\Z))$ for $\sigma>1$.
See, Lemma 3.2 of \cite{CT09SIAM} for the existence of such limit.
\item
For $j=1,2$, we define
$\phi_{j,R}:=\phi_j$ and $\phi_{j,I}:=\im \phi_j.$
\item
We set $z_{j,R}:=\Re z_j$, $z_{j,I}:=\Im z_j$ and $D_{j,A}=\partial_{z_{j,A}}$ for $j=1,2$ and $A=R,I$.
\item
We set $P_cu:=u-\sum_{j=1,2A=R,I}\<u,\phi_{j,A}\>\phi_{j,A}.$

\end{itemize}

It is well known that there exist families of  small nonlinear bound states of \eqref{1} which bifurcate from $\phi_j$.
For the proof, see \cite{MDNLS1}.

\begin{proposition}\label{prop:1}

Fix $j\in \{1,2\}$.
There exist $a_0>0$ and $\delta_0>0$ s.t.\ for all $z\in B_{\C}(\delta_0)$, there exists $\tilde e_j\in C^\omega\(B_{\R}(\delta_0^2); \R\)$ and $q_j\in C^\omega\(B_{\R}(\delta_0^2); l^{a_0}_e(\Z;\R)\)$ s.t.\ $\<\phi_j, q_j\>=0$ and 
\begin{align}\label{3}
\phi_j(z):=z\tilde\phi_j(|z|^2)=z\(\phi_j+q_j(|z|^2)\),
\end{align} satisfies
\begin{align}\label{4}
\(H-E_j(|z|^2)\)\phi_j(z)+\beta(|\phi_j(z)|^2)\phi_j(z)=0,
\end{align}
where $E_j(|z|^2)=e_j+\tilde e_j(|z|^2)$.
Further, we have $|\tilde e_j(|z|^2)|+\|q_j(|z|^2)\|_{l_e^{a_0}}\lesssim |z|^6$.
\end{proposition}

Using the nonlinear bound states, we can express arbitrary $u\in l^2$ with $\|u\|_{l^2}\ll1$ such as
\begin{align*}
u=\phi_1(z_1)+\phi_2(z_2)+ R[z]\eta,
\end{align*}
where $z_1,z_2\in \C$, $\eta\in P_c l^2$ and $R[z]=R[z_1,z_2]$ is some near identity operator (see Lemma \ref{lem:2} and Lemma \ref{lem:3}).
Thus, the study of the dynamics of $u$ will reduce to the study of the system of ODE and PDE which governs $z_1,z_2$ and $\eta$.

By a normal form argument, we can simplify the ODE-PDE system as follows.
\begin{proposition}\label{normalform}
There exists a transformation in the neighborhood of the origin of $\C^2\times P_cl^2$ such that the new coordinate $(\tilde z_1, \tilde z_1, \tilde \eta)$ and the original coordinate $(z_1,z_2,\eta)$ is near in the following sense:
\begin{align}\label{nearid}
|z-\tilde z| + \|\eta - \tilde \eta\|_{l_e^{a_{N_0}}}\lesssim |z|^5 \(|z_1z_2| + \|\eta\|_{l_e^{-b_{N_0}}}\),
\end{align}
where $b_{N_0}>0$ is a constant.
Moreover, the new coordinate (which will just write $(z_1,z_2,\eta)$) satisfies the following system: 
\begin{align}
\im \dot z_1 &= e_1 z_1 + A_1(|z_1|^2,|z_2|^2)z_1+(N_0-1)\bar z_1^{N_0-2}z_2^{N_0}(G, \eta)+\mathcal R_1,\label{intro41}\\
\im \dot z_2 &= e_2 z_2 + A_2(|z_1|^2,|z_2|^2  )z_2+N_0z_1^{N_0-1}\bar z_2^{N_0-1}(\bar G , \bar \eta)+\mathcal R_2,\label{intro42}\\
\im \eta_t &= H\eta + P_c \beta(|\eta|^2  )\eta+ \bar z_1^{N_0-1}z_2^{N_0}G  +  \mathcal R_{\eta},\label{intro43}
\end{align}
where $\mathcal R_1,\mathcal R_2,\mathcal R_\eta$ are higher order error terms, $A_1,A_2$ are $\R$-valued functions and $G\in l_e^{b_{N_0}}$.
\end{proposition}

\begin{remark}
The near identity transformation in Proposition \ref{normalform} is given by the composition of transformation given in Propositions \ref{prop:dar} and \ref{prop:birk}.
Here, $b_{N_0}$ given in \eqref{nearid} will be $a_{2N_0}$ in Proposition \ref{prop:birk} because we use Proposition \ref{prop:birk} with $M=2N_0$. 
\end{remark}

\begin{remark}
The estimate \eqref{nearid} ensures us that if $|z_1z_2|\to 0$ and $\|\eta\|_{L_e^{a_{N_0}}}\to 0$ (which we will show in our main theorem), then the original coordinate and the new coordinate corresponds.
Therefore, we can work on the new coordinate only to get our result.
\end{remark}

\begin{remark}
$G$ in Proposition \ref{normalform} will corresopond to $G^{2N_0}_{N_0-1,2,0}(0)$ in Proposition \ref{prop:birk}.
\end{remark}

For our result, we need a nondegeneracy condition related to $G$ which appears in the system \eqref{intro41}--\eqref{intro42}.
We will assume the following Fermi Golden Rule assumption
\begin{align}\label{FGR}
\Gamma:=\Im (R_H^+(\omega_{N_0})G,G)>0.\tag{FGR}
\end{align}
We note that $\Gamma\geq 0$ in general.
So, the assumption is that $\Gamma\neq 0$.

Our main result is the following.
\begin{theorem}\label{thm:1}
Assume \eqref{2} and $(\mathrm{FGR})$.
Then, there exists $\delta>0$ s.t.\ if $\|u(0)\|_{l^2}<\delta$, there exists $j\in \{1,2\}$, $z\in C^1(\R;\C)$, $\rho_+>0$ and $\eta_+\in l^2$ s.t.
\begin{align}
&\lim_{t\to\infty}\|u(t)-\phi_j(z(t))-e^{\im t \Delta}\eta_+\|_{l^2}=0,\label{5}\\
&\lim_{t\to \infty}|z(t)|\to \rho_+,\label{6}
\end{align}
and $\|\eta_+\|_{l^2}+\rho_+\lesssim \|u(0)\|_{l^2}$, where $u(t)$ is the solution of \eqref{1} with $\lambda=\lambda_0$.
\end{theorem}

\begin{remark}
The equation \eqref{5} in the statement of Theorem \ref{thm:1} shows that the solution $u(t)$ can be expressed as $u(t)=\phi_j(z(t))+e^{\im t \Delta}\eta_+ + \mathrm{error}(t)$, where $\|\mathrm{error}(t)\|_{l^2}\to 0$ as $t\to \infty$.
Moreover, since $e^{\im t \Delta}\eta_+ $ vanishes in any compact domain as $t\to \infty$, the solution locally (in space) converges to $\phi_j(z(t))$.
\end{remark}

\begin{remark}
We note that, $j$($\in \{1,2\}$) in Theorem \ref{thm:1} depends on the initial data $u(0)$.
Therefore, even if one may get the impression that both $\phi_1$ and $\phi_2$ are stable, it is not the case.
Indeed, we will show that $\phi_2$ is unstable (Theorem \ref{thm:2}).
Therefore, we expect that for generic initial data (where we do not have the precise definition of "generic"), the solutions converge to $\phi_1$ and only for some exceptional initial data, the solutions converge to $\phi_2$.
\end{remark}

\begin{remark}
For given $G$, $\Gamma$ can be expressed as
\begin{align}\label{FGRexpression}
\Gamma=\frac{\pi}{4\sin (\xi_{N_0})} \(|\hat G(\xi_{N_0})|^2+|\hat G(- \xi_{N_0})|^2\),
\end{align}
where $\xi_{N_0}=\arccos (\frac{1}{2}(2-\omega_{N_0}))$ and $\hat G$ is the distroted Fourier transform of $G$ associated to $H$ (see \cite{Cuccagna09JMAA}).
We will give the proof of this formula in the appendix of this paper.
Now, the assumption \eqref{FGR} reduces to the condition
\begin{align*}
\hat G(\xi_{N_0})\neq 0\text{ or }\hat G(-\xi_{N_0})\neq 0.
\end{align*}
\end{remark}
\begin{remark}
For $N_0=4$, which will be the simplest case in our situation, $G$ will given by
\begin{align*}
G=\phi_1^3\phi_1^4.
\end{align*}
Clearly seen by the above expression, $\Gamma$ is related to the overlap of the two eigenvalues of $H$.
\end{remark}
\begin{remark}\label{rem:N023}
Unfortunately, for the cases $N_0=2,3$, $G$ will be $0$ (and so $\Gamma=0$) due to the fact that the nonlinearity has no cubic and quintic term.
However, one can still assume a generalized version of Fermi Golden Rule assumption such as \cite{CuMaAPDE} and obtain the same result in Theorem \ref{thm:1} as well as Theorems \ref{thm:2}, \ref{thm:3} with some modification of the proof.
In these case, we will have to take into account the higher order terms and in particular, $G$ appearing in \eqref{intro41}--\eqref{intro43} will have to be modified as
\begin{align*}
G=6|z_1|^4 \phi_1^5\phi_2^2+12 |z_1|^2|z_2|^2\phi_1^3\phi_2^4+3|z_2|^4\phi_1\phi_2^6,
\end{align*}
for the case $N_0=2$ and 
\begin{align*}
G=4|z_1|^2\phi_1^4\phi_2^3+3|z_2|^2\phi_1^2\phi_2^5,
\end{align*}
for the case $N_0=3$.
The assumption will now be
\begin{align*}
\Gamma:=\Im(R_H^+(\omega_{N_0})G,G)\geq C\times\begin{cases} |z_1|^8+|z_2|^8 & N_0=2\\ |z_1|^4+|z_2|^4 & N_0=3\end{cases},
\end{align*}
for some constant $C>0$.
\end{remark}

In this paper we also prove several results which give deeper understanding to the dynamics of small solutions of DNLS \eqref{1}.
In particular, we show
\begin{itemize}
\item
the orbital instability of excited state $\phi_2(z)$ (Theorem \ref{thm:2}),
\item
the generalization of equipartition property proved by Gang--Weinstein \cite{GW11AMRX} (Theorem \ref{thm:3}).
\end{itemize}

We say that a nonlinear bound state $\phi$ is orbitally stable if 
$$ \forall \varepsilon>0,\ \exists \delta>0\ \mathrm{s.t.}\ \text{if}\ \|u(0)-\phi\|_{l^2}< \delta,\ \text{then}\ \sup_{t>0}\inf_{\theta}\|u(t)-e^{\im \theta}\phi\|_{l^2}<\varepsilon.$$
If $\phi$ is not orbitally stable, we say $\phi$ is orbitally unstable.
We say that a nonlinear bound state $\phi$ is a ground state if $E(\phi)=\inf\{E(\psi)\ |\ \|\psi\|_{l^2}=\|\phi\|_{l^2},\ \psi\text{ is a nonlinear bound state}\}$, where $E$ is the energy of DNLS \eqref{1} given in \eqref{13.1} (there are many definitions of ground state, we adopt this definition to make the following discussion clear).
Nonlinear bound states which are not ground states will be called excited states.
In this sense, $\phi_1(z)$ are ground states and $\phi_2(z)$ are excited states for $|z|\ll1$.
It is a classical result by Rose--Weinstein \cite{RW88PD} that under our assumption all ground states $\phi_1(z)$ with $|z|\ll1$ are orbitally stable (see also \cite{FO03DIE}).
On the other hand, the orbital stability/instability of excited states are a subtle problem and there are not many rigorous results (see the discussion in \cite{Cuccagna09PhysD, CM16JNS, KPS15PRL,Mizumachi07ADE}).
In fact, one should notice that excited states of linear Schr\"odinger equation are orbitally stable (See also \cite{MM13DIE}).
Further, it was shown by the author \cite{MDNLS1} that if we have \eqref{2}, then the excited states $\phi_2(z)$ with $|z|\ll1$ are orbitally stable.
However, if we have \eqref{2.1} and assume (FGR), then excited states turn out to be orbitally unstable.
This result corresponds to Theorem 1.4 of Cuccagna--Maeda \cite{CuMaAPDE}.

\begin{theorem}\label{thm:2}
Under the assumption of Theorem \ref{thm:1}, $\phi_2(z)$ is orbitally unstable.
\end{theorem}

By Theorem \ref{thm:1}, we see that only one of the nonlinear bound state is selected and the other disappears (and by Theorem \ref{thm:2}, usually a ground state is selected).
Therefore, it is natural to ask that what amount of mass ($l^2$ norm) of the excited state will be transported to the ground state and what amount will be damped to spatial infinity.
The answer is quite surprising.
In \cite{GW11AMRX} Gang--Weinstein proved that, for continuous NLS with two eignevalues with $N_0=2$, the excited state component is divided approximately half and half. 
That is, half of the mass is damped to the spatial infinity and the other half is absorbed to the ground state.
Because of this fact, Gang--Weinstein \cite{GW11AMRX} called this phenomenon ``equipartition property".

Here, we generalize Gang--Weinstein's result (although we are considering DNLS, the same proof holds for continuous NLS with two eigenvalues).
In particular, we consider the cases for arbitrary $N_0\geq 2$ and also the case which excited states are selected.

\begin{theorem}\label{thm:3}
Under the assumption and conclusion of Theorem \ref{thm:1}, set $\varepsilon:=\|u(0)\|_{l^2}<\delta$, where $\delta$ is given in Theorem \ref{thm:1}.
Then, if $u(t)$ converges to $\phi_1(z)$, we have
\begin{align*}
\rho_+=|(u(0),\phi_1)|^2+ \frac{N_0-1}{N_0} |(u(0),\phi_2)|^2 + O(\varepsilon^4),
\end{align*}
and if 
$u(t)$ converges to $\phi_2(z)$, we have
\begin{align*}
\rho_+=\frac{N_0}{N_0-1} |(u(0),\phi_1)|^2 +|(u(0),\phi_2)|^2  + O(\varepsilon^4). 
\end{align*}
where $\rho_+$  is  given in Theorem \ref{thm:1}.
\end{theorem}

\begin{remark}
For the case $N_0=2$, if $u(t)$ converges to the ground state $\phi_1(z)$, we have $|z(t)|^2\to |(u(0),\phi_1)|^2+\frac{1}{2} |(u(0),\phi_2)|^2+O(\varepsilon^4)$, which is the equipartition property of Gang--Weinstein \cite{GW11AMRX}.
However, we note that the $N_0=2$ case as well as $N_0=3$ case need generalized version of \eqref{FGR} given in remark \ref{rem:N023}.

\end{remark}

The proof of Theorems \ref{thm:1}, \ref{thm:2} and \ref{thm:3} are based on the argument developed in \cite{CuMaAPDE}.
Following standard arguments, we first decompose the solution in the form $u(t)=\phi_1(z_1(t))+\phi_2(z_2(t))+\eta$, where $\eta$ satisfying suitable orthogonal conditions.
This will be done in section \ref{sec:Coordinates}.
By such decomposition, we reduce DNLS \eqref{1} into a system of two complex ODE and one DNLS-like PDE.
However, this system will be very complicated.
To simplify the system and moreover to be able to apply Birkhoff normal form argument, we will perform the first change of coordinate to make the coordinate to be ``canonical" (or in other words, diagonalize the symplectic form).
This is done by Darboux theorem (Proposition \ref{prop:dar}).
We next apply the Birkhoff normal form argument (Proposition \ref{prop:birk}), which is another change of coordinate, developed in \cite{Bambusi13CMPas,BC11AJM,  Cuccagna11CMP,Cuccagna12Rend, Cuccagna14TAMS, CuMaAPDE}.
By Birkhoff normal form argument, we can change the coordinate $(z_1,z_2,\eta)$ s.t.\ DNLS \eqref{1} will be a Hamiltonian equation with the new Hamiltonian $E_{\mathrm{eff}}+R$, where $R$ is the remainder.
Here, the effective Hamiltonian becomes something like
\begin{align}\label{Eeff}
E_{\mathrm{eff}}(z_1,z_2,\eta):=\frac{1}{2}\sum_{j=1,2}e_j|z_j|^2+A(|z_1|^2,|z_2|^2)+E(\eta)+\<\bar z_1^{N_0-1}z_2^{N_0}G,\eta\>,
\end{align}
where $G$ is a Schwartz function and $(z_1,z_2,\eta)\in \C\times\C\times l^2(\Z)$.
The energy in the original coordinate (see \eqref{14}) will have many terms with both resonant and nonresonant frequencies.
One can think each $z_1$ has frequency $e^{-\im e_j t}$ so the frequency of a monomial $z_1^{\mu_1}z_2^{\mu_2}\bar z_1^{\nu_1}\bar z_2^{\nu_2}$ is $e^{-\im\(e_1(\mu_1-\nu_1)+e_2(\mu_2-\nu_2)\)t}$.
So, if $e_1(\mu_1-\nu_1)+e_2(\mu_2-\nu_2)\notin[0,4]$, a first order in $\eta$ term, which is responsible to the interaction between $z$ and $\eta$, in the energy with the form $\<z_1^{\mu_1}z_2^{\mu_2}\bar z_1^{\nu_1}\bar z_2^{\nu_2} \tilde G,\eta\>$ (with some Schwartz function $\tilde G$) can be regarded as a nonresonant term and if $e_1(\mu_1-\nu_1)+e_2(\mu_2-\nu_2)\in(0,4)$, then such term is a resonant term.
The role of the Birkhoff normal form is to erase the nonresonant terms.
The last term in \eqref{Eeff} is the resonant term of the lowest order, which dominates all the other resonant terms.
By such procedure, we will arrive to the system \eqref{intro41}--\eqref{intro43}, which
 is similar to the "nonlinear toy model" of Weinstein \cite{WeinsteinSurvey}.
For the precise form of the Hamiltonian and the system, see \eqref{35} and \eqref{41}--\eqref{43}.
Now, if we set $z_j(t)=e^{-\im e_j t}z_j(0)$, then $Y:=-\bar z_1^{N_0-1}z_2^{N_0}R^+_H(\omega_{N_0})G$ becomes a solution of the third equation without the nonlinear term $P_c\beta(|\eta|^2)\eta$.
Thus, substituting $\eta= Y+``\mathrm{error}"$ to the equations of $z_j$, we obtain
\begin{align*}
&\frac 1 2 \frac{d}{dt}|z_1|^2 = (N_0-1)\Gamma|z_1|^{2(N_0-1)}|z_2|^{2N_0}+\mathrm{error},\quad
\frac 1 2 \frac{d}{dt}|z_2|^2 = -N_0 \Gamma |z_1|^{2(N_0-1)}|z_2|^{2N_0}+\mathrm{error},
\end{align*}
where $\Gamma:=-\mathrm{Im}( G, R_H^+(\omega_*)G)\geq 0$.
Then, by integrating (say) the second equation, provided $\Gamma>0$, we obtain the integrability of $|z_1|^{2(N_0-1)}|z_2|^{2N_0}$.
The assumption $\Gamma>0$ is the assumption (FGR).

Theorem \ref{thm:1} is a consequence of the above argument combined with the Strichartz and Kato smoothing estimates.
Further, Theorem \ref{thm:2} will be a easily deduced from Theorem \ref{thm:1} combined with simple observation of the energy of the initial data and the final data.
Next, notice that the Effective Hamiltonian $E_{\mathrm{eff}}$ is invariant under $(z_1,z_2,\eta)\mapsto (e^{\im N_0\theta}z_1,e^{\im (N_0-1)\theta}z_2,\eta)$.
Therefore, we have an ``almost" conservation of $N_0|z_1(t)|^2+(N_0-1)|z_2(t)|^2$.
Comparing the initial data and the final data, we will arrive to the generalize equipartition property (Theorem \ref{thm:3}).

This paper is organized as follows.
In section \ref{sec:mod}, following \cite{GNT04}, we introduce a nonlinear coordinate by a standard modulation argument.
In section \ref{sec:normal}, we introduce Darboux theorem and Birkhoff normal form arguments.
In section \ref{sec:dispersion}, we introduce some linear estimates and give estimates for the solution of \eqref{1} in Strichartz and weighted spaces by Bootstrap argument.
In section \ref{sec:proofmain}, we give the proof of Theorems \ref{thm:1}, \ref{thm:2} and \ref{thm:3}.
In section \ref{sec:prooftech}, we gathered the proofs of Darboux theorem (Proposition \ref{prop:dar}), Birkhoff normal form (Proposition \ref{prop:birk}) and a local decay estimate (Lemma \ref{lem:l4}).
This section will be technical.

\section{Nonlinear coordinates}\label{sec:mod}

In this section, we introduce the nonlinear coordinate by standard modulation argument.
Further, we expand the energy with respect to this coordinate.
\subsection{Coordinates}\label{sec:Coordinates}

We first decompose $u$ as a sum of nonlinear bound states and a function in $\mathcal H_c[z   ]$.

\begin{definition}
We set $l^2_c:=P_c l^2$.
Further, for $a\geq -a_0$, we set $l_{e,c}^{a}:=P_cl_e^a$.
Notice that $\phi_{j,A}\in l_e^{a_0}$, we can extend $P_c$ to $l_e^{-a_0}$.
\end{definition}

\begin{lemma}\label{lem:1}
There exists $\delta>0$ s.t.\ 
there exists $z=(z_1,z_2)\in C^\omega(B_{l^2}(\delta)  ;\C^2)$ s.t.
\begin{align*}
v(u   ):=u-\phi_1(z_1(u   )   )-\phi_2(z_2(u   )   )\in \mathcal H_c[z(u   )].
\end{align*}
\end{lemma}

\begin{proof}
The proof is standard.
Set
\begin{align*}
\mathcal F(u,z   ):=(\mathcal F_{1,R}, \mathcal F_{1,I},\mathcal F_{2,R},\mathcal F_{2,I}),
\end{align*}
where
\begin{align*}
\mathcal F_{j,A}(u,z   ):=\<\im \(u-\phi_1(z_1)-\phi_2(z_2 )\), D_{j,A}\phi_j(z_j   )\>
\end{align*}
Then, the conclusion follows from the implicit function theorem and the analyticity of $\mathcal F$ with respect to $u,z,\lambda$.
See Lemma 3.1 of \cite{MDNLS1}.
\end{proof}

Since the nonlinear continuous space $\mathcal H_c[z   ]$ depends on $z$ which depends on $u$, it varies when $u$ varies.
To fix the space where $v$ belongs, we introduce $R[z   ]:l^2_c\to \mathcal H_c[z   ]$ which is an inverse of $P_c$ restricted on $\mathcal H_c[z   ]$.
\begin{lemma}\label{lem:2}
There exists $\delta>0$ s.t.\ there exists $\alpha_{j,A}\in C^\omega (B_{\C^2}(\delta);l_e^{a_0}(\Z;\C))$ for $j=1,2$ and $A=R,I$, s.t.\ 
$
\|\alpha_{j,A}(z   )\|_{l_e^{a_0}}\lesssim|z|^6,
$
Further, 
\begin{align}\label{10}
R[z   ]\eta=\eta + \sum_{j=1,2,A=R,I}\<\alpha_{j,A}(z   ), \eta\>\phi_{j,A}.
\end{align}
satisfies $R[z   ]:l_c^2\to \mathcal H_c[z   ]$ and $\left.P_c\right|_{\mathcal H_c[z   ]}=R[z   ]^{-1}$.
\end{lemma}

\begin{proof}
The proof is standard.
See, for example \cite{MDNLS1}.
\end{proof}

Combining Lemmas \ref{lem:1} and \ref{lem:2}, we have a coordinate in $B_{l^2}(\delta)$ for sufficiently small $\delta>0$.
\begin{lemma}\label{lem:3}
Let $\delta>0$ sufficiently small.
Then 
\begin{align}\label{11}
\mathcal F(z,\eta   ):=\sum_{j=1,2}\phi_j(z_j   )+R[z   ]\eta \in C^\omega(B_{\C^2\times l_c^2}(\delta)  ;l^2),
\end{align}
is a $C^\omega$ diffeomorphism to the $l^2$ neighborhood of the origin.
Further, we have
\begin{align}\label{12}
|z_1|+|z_2|+\|\eta\|_{l^2}\sim \|\mathcal F(z,\eta   )\|_{l^2}.
\end{align}
\end{lemma}

\begin{remark}
Notice that
\begin{align}\label{13}
\psi(z,\eta:\lambda):=\phi_1(z_1)+\phi_2(z_2)+(R[z]-1)\eta \in C^\omega(B_{\C^2\times l_{e,c}^{-a_0}}(0,\delta)  ;l_{e}^{a_0}).
\end{align}
\end{remark}

We will define $\eta(u)$ from \eqref{11}.
\begin{definition}\label{def:eta}
Let $\delta>0$ sufficiently small.
We set $\eta\in C^\omega(B_{l^2}(\delta)  ;l^2_c)$ by
\begin{align}\label{eq:defeta}
\eta(u):= P_c \(u-\sum_{j=1,2}\phi_j(z_j(u)   )\),
\end{align}
where $z_j$ are given by Lemma \ref{lem:1}.
\end{definition}

\subsection{Expansion of energy}

It is well known that \eqref{1} conserves the $l^2$-norm and the energy:
\begin{align}\label{13.1}
E(u   ):=\frac 1 2 \< Hu,u\> + \frac 1 2 \sum_{n\in\N}B(|u(n)|^2   ),
\end{align}
where
\begin{align*}
B(s   ):=\int_0^s\beta(s   )\,ds=\frac 1 4 \|u\|_{l^8}^8+\sum_{j=4}^M \frac{\lambda_j}{2 j+2} \|u\|_{l^{2j+2}}^{2j+2}.
\end{align*}
We set $E^0(z,\eta   ):=\mathcal F^* E(z,\eta   ):=E(\mathcal F(z,\eta   )   )$.
Our interest here is the expansion of $E^0$ with respect to $z$ and $\eta$.
However, before that we introduce a notation mainly to represent the remainder terms.

\begin{definition}\label{def:remainder}
We set
\begin{align*}
\mathcal R_X(a,\delta):= C^\omega(B_{\C^2\times l_{e,c}^{-a}}(\delta)  ;X),
\end{align*}
where $X=\R,\C,\C^2$, $l_e^a$ and $l_{e,c}^a$.
\end{definition}

Since (one of) our aim is to show $|z_1z_2|\to 0$, we set 
\begin{align}\label{def:Z}
Z:=Z(z):=z_1\bar z_2,
\end{align}
and count how many $Z$ is there in each terms of the energy expansion.

\begin{proposition}[Energy expansion]\label{prop:expansion1}
There exists $a>0$ s.t.\ 
for all $k\geq 1$, $0\leq l\leq k$ and $j=0,1,2$, there exist $C_{k,j,l}\in C^\omega(B_{\R}(0,\delta_0^2)  ;\R)$ and for all $k\geq 0$, $0\leq l \leq k$ and $j=1,2$, there exist $G_{k,j,l}\in C^\omega(B_{\R}(0,\delta_0^2)  ;l_e^{a})$ s.t.
\begin{align}
E^0(z,\eta   )=&E_1(|z_1|^2   )+E_2(|z_2|^2   )+E^0(0,\eta   )\label{14}\\&+\sum_{k\geq 1} \sum_{j=0,1,2}\sum_{0\leq l\leq k}C_{k,j,l}(|z_j|^2   )Z^{l}\bar Z^{k-l} + \sum_{k\geq 0}\sum_{j=1,2}\sum_{0\leq l\leq k}\<Z^l\bar Z^{k-l}z_jG_{k,j,l}(|z_j|^2   ),\eta\>\nonumber\\&+ R(z,\eta   ),\nonumber
\end{align}
where $z_0\equiv 0$, $R\in \mathcal R_\R(a,\delta_0)$ with $|R(z,\eta   )|\lesssim |z|\(|z|+\|\eta\|_{l_e^{-a}}\)^5\|\eta\|_{l_e^{-a}}^2$ and $E_j(|z_j|^2   )=E(\phi_j(z_j   ))$.
Further, we have $\overline{C_{k,j,l}}=C_{k,j,k-l}$ and $C_{1,j,l}(0   )=0$ for $l=0,1,j=0,1,2$.
\end{proposition}

\begin{proof}
First, set $\psi(z,\eta   ):=\phi_1(z_1)+\phi_2(z_2)+\sum_{j=0,1,A=R,I}\<\alpha_{j,A}(z   ),\eta\>\phi_{j,A}$.
Then, we have $\mathcal F(z,\eta   )=\eta+\psi$ and $\psi\in \mathcal R_{l_e^{a_0}}(a_0,\delta)$ and $|\psi|\lesssim |z|$ (see \eqref{13}).
Now, by Taylor expansion, we have
\begin{align}\label{20}
&E(\eta+\psi(z,\eta   )   )=E(\eta   )+\int_0^1\<\nabla E(\eta+ s \psi   ), \psi\>\,ds.
\end{align}
Notice that $E(\eta   )=E^0(0,\eta   )$.
We now show that the second term in the r.h.s.\ of \eqref{20} is $\mathcal R_\R(a,\delta_0)$ for some $a>0$.
Expanding the second term of the r.h.s.\ of \eqref{20}, we have
\begin{align*}
\int_0^1\<\nabla E(\eta+ s \psi   ), \psi\>\,ds=
\int_0^1\<H(\eta+s\psi),\psi\>\,ds+\int_0^1\< \beta(|\eta+s\psi|^2)(\eta+s\psi),\psi\>.
\end{align*}
The first term is obviously in $\mathcal R_\R(a_0,\delta_0)$ and bounded by $\lesssim |z|$.
For the second term, first notice that $$\int_0^1|\eta+s\psi|^{2n}(\eta+s\psi)\,ds\in \mathcal R_{l^2}(0,\delta),$$
for $n\geq 0$.
Thus, for any $a>0$, we have
\begin{align*}
\int_0^1 |e^{-a|\cdot|}(\eta+s\psi)|^{2n}e^{-a|\cdot|}(\eta+s\psi)\,ds \in \mathcal R_{l^2}(a,\delta).
\end{align*}
Now, since
\begin{align*}
\<\int_0^1 |\eta+s\psi|^{2n}(\eta+s\psi)\,ds,\psi\>=\<\int_0^1 |e^{-\frac{a_0}{(2n+1)}|\cdot|}(\eta+\psi)|^{2n}e^{-\frac{a_0}{(2n+1)}|\cdot|}(\eta+\psi), e^{a_0 |\cdot|}\psi\>,
\end{align*}
we see that the above term belongs to $\mathcal R_\R(a_0/(2n+1),\delta)$.
Therefore, we have
\begin{align*}
\int_0^1 \<\nabla E(\eta+s \psi),\psi\>\,ds\in \mathcal R_\R(a,\delta_0),
\end{align*}
with $a=a_0/(2M+1)$(Recall \eqref{betapoly}).
Thus, expanding $\int_0^1 \<\nabla E(\eta+s \psi),\psi\>\,ds$ with respect to $z_1,z_2,\eta$ and rearranging them, we have the expansion \eqref{14}.

Finally, the property $\overline{C_{k,j,l}}=C_{k,j,k-l}$ comes from the fact that $E$ is real valued and the property $C_{1,j,l}(0)=0$ comes from the fact that the only possible source of such term are $\<H z_j\phi_j, z_{3-j}\phi_{3-j}\>$ and they are $0$.
\end{proof}

\section{Normal form argument}\label{sec:normal}

In this section, we change of the coordinate given in Lemma \ref{lem:3} by Darboux theorem (Proposition \ref{prop:dar}) and Birkhoff normal form argument (Proposition \ref{prop:birk}).
The proof will given in sections \ref{sec:proofdar} and \ref{sec:proofbirk}.
As explained in the introduction, the role of Darboux theorem is to diagonalize the symplectic form and make it possible to proceed the Birkhoff normal form argument.
Next, by the Birkhoff normal form argument we erase the nonresonant terms.

\subsection{Darboux theorem}
Set $\Omega(X,Y):=\<\im X, Y\>$.
We define a new symplectic form $\Omega_0$ by
\begin{align*}
\Omega_0(X,Y):=&\sum_{j=1,2}\Omega(d \phi_j(z_j)X, d\phi_j(z_j)Y)+\Omega(d\eta X,d\eta Y)\\=&\sum_{j=1,2}\frac{\im}{2}(1+\gamma_j(|z_j|^2))dz_j\wedge d\bar z_j (X,Y) + \Omega(d\eta X,d\eta Y),
\end{align*}
where $dz_j,d\eta$ are the Fr\'echet derivative of $z_j, \eta$ given in Lemma \ref{lem:1} and Definition \ref{def:eta} respectively, $$dz_j\wedge d \bar z_j(X,Y)=X_j Y_{\bar j}-X_{\bar j}Y_j$$ with $X_{j}=dz_j(X)$, $X_{\bar j}=d\bar z_j (X)$
and
$$\gamma(|z_j|^2)=\<\tilde q + |z_j|^2 \tilde q' (|z_j|^2),\tilde q + |z_j|^2 \tilde q' (|z_j|^2)\>+|z_j|^4\<\tilde q'(|z_j|^2),\tilde q'(|z_j|^2)\>.$$
Note that $\Omega$ is the symplectic form associated to the Hamilton equation \eqref{1}.
We want to change $\Omega$ to $\Omega_0$, which has no cross terms.

\begin{proposition}\label{prop:dar}
For $a>0$ given in Proposition \ref{prop:expansion1}, there exists $\delta>0$ s.t.\ there exists $\tilde{\mathcal Y}^D\in \mathcal R_{l_e^{a}}(a,\delta)$ s.t. $\mathcal Y^D:=\mathrm{Id}+\tilde{\mathcal Y}^D$ is a $C^\omega$ diffeomorphism and $\mathcal Y^*\Omega=\Omega_0$.
Further, setting $z^D(z,\eta   ):=(\mathcal Y^D)^*z$ and $\eta^D(z,\eta   ):=(\mathcal Y^D)^*\eta$, we have
\begin{align}\label{24}
|z^D - z| + \|\eta^D-\eta\|_{l_e^{a}} \lesssim |z|^6  \|\eta\|_{l_e^{-a}}+|z|^5|z_1z_2|.
\end{align}
\end{proposition}

The proof is similar to the proof of Darboux theorem in \cite{MDNLS1} (see also \cite{CuMaAPDE}).
However, for convenience of the readers, we give the proof in section \ref{sec:proofdar}

Let $F\in C^1(B_{l^2}(\delta);\R)$.
Then the Hamiltonian vector field $X_F$ with respect to the symplectic form $\Omega_0$ is defined by the relation $i_{X_F}\Omega_0 = dF$.
Comparing
\begin{align*}
&\Omega_0(X_F,Y)=\frac \im 2 \sum_{j=1,2}(1+\gamma_j(|z_j|^2))\((X_F)_{j}Y_{\bar j}-(X_F)_{\bar j}Y_j\)+\Omega ((X_F)_\eta, Y_\eta),
\end{align*}
and
\begin{align*}
\<\nabla F, Y\>=\partial_j F Y_j +\partial_{\bar j} F Y_{\bar j} + \<\nabla_\eta F, Y_\eta\>,
\end{align*}
we have
\begin{align}
&(X_F)_j = -2\im(1+\tilde \gamma_j(|z_j|^2))\partial_{\bar j}F,\quad (X_F)_{\bar j} =2\im (1+\tilde \gamma_j(|z_j|^2))\partial_{j}F,\label{25}\\&
(X_F)_\eta :=d\eta (X_F)= -\im \nabla_\eta F,\label{26}
\end{align}
where $\tilde \gamma_j(|z_j|^2)$ is defined by $(1+\gamma_j(|z_j|^2))^{-1}=1+\tilde \gamma_j(|z_j|^2)$.
In the following, we set $E^D(z,\eta   ):=(\mathcal Y^D)^*E(z,\eta   )$.

\begin{proposition}[Energy expansion in the Darboux coordinate]\label{prop:expansionD}
There exist $a_D>0$ and $\delta>0$ s.t.\ 
for all $k\geq 1$, $0\leq l\leq k$ and $j=0,1,2$, there exist $C_{k,j,l}^D\in C^\omega(B_{\R}(\delta^2)  ;\C)$ and for all $k\geq 0$, $0\leq l \leq k$ and $j=0,1$, there exist $G_{k,j,l}^D\in C^\omega(B_{\R}(\delta^2)  ;l_e^{a_D})$ s.t.
\begin{align*}
E^D(z,\eta   )=&E_1(|z_1|^2   )+E_2(|z_2|^2   )+E^0(0,\eta   )\\&+\sum_{k\geq 2}\sum_{j=0,1,2} \sum_{0\leq l\leq k}C_{k,j,l}^D(|z_j|^2 )Z^{l}\bar Z^{k-l} + \sum_{k\geq 1}\sum_{j=1,2}\sum_{0\leq l\leq k}\<Z^l\bar Z^{k-l}z_jG_{k,j,l}^D(|z_j|^2   ),\eta\>\\&+ R_D(z,\eta   ),
\end{align*}
where $R_D\in \mathcal R_\R(a_D,\delta)$ (recall Definition \ref{def:remainder}) with $|R(z,\eta   )|\lesssim |z|(|z|+\|\eta\|_{l_e^{-{a_D}}})^5\|\eta\|_{l_e^{-a_D}}^2$
Further, since $E$ is real valued, we have $\overline{C_{k,j,l}}^D=C_{k,j,k-l}^D$.
\end{proposition}

\begin{proof}
First, notice that we can extend $\mathcal Y^D$ in $\mathcal R_{l_e^{-a}}(a,\delta)$ for $a>0$ given in Proposition \ref{prop:dar}.
Next, by the proof of proposition \ref{prop:expansion1}, we have $E^0(z,\eta   )-E^0(0,\eta   )\in \mathcal R_\R(a,\delta)$.
Therefore, we see that $E^0(z^D,\eta^D   )-E^0(0,\eta^D   )$ has the expansion similar as \eqref{14} without the term $E^0(0,\eta   )$.
Further, since $\eta^D=\eta+\tilde \eta^D$ with $\tilde \eta^D\in \mathcal R_{l_{e,c}^{a}}(a,\delta)$, the same argument as in the proof of Proposition \ref{prop:expansion1} holds for the expansion of $E(0,\eta+\tilde \eta^D   )$ and as a conclusion we have the same expansion as \eqref{14}.

The only nontrivial part is the absence of terms
\begin{align*}
\sum_{j=0,1,2} \sum_{0\leq l\leq 1}C_{1,j,l}^D(|z_j|^2   )Z^{l}\bar Z^{1-l} + \sum_{j=1,2}\<z_jG_{0,j,0}^D(|z_j|^2   ),\eta\>.
\end{align*}
We first show $G_{0,1,0}^D=0$.
Starting from initial data $(z_1(0),z_2(0),\eta(0))=(z_1(0),0,0)$, by \eqref{26}, we have
\begin{align*}
\left.\im \eta_t\right|_{t=0}=z_1(0)G_{0,1,0}^D(|z_1(0)|^2   ).
\end{align*}
On the other hand, because of \eqref{24}, $(z_1(0),0,0)$ is invariant under the transformation $\mathcal Y^D$.
Thus, as in the original coordinate, $(z_1(0),0,0)$ is the initial data of nonlinear bound state.
Consequently, we have $\im \eta_t=0$.
Therefore, $G_{0,1,0}^D=0$.
Similarly, we have $G_{0,2,0}^D=0$.

We next show $C_{1,j,l}^D=0$ for all $l=0,1,\ j=0,1,2$.
First, notice that by Proposition \ref{prop:expansion1} and \eqref{24}, $C_{1,j,l}^D(0   )=C_{1,j,l}(0   )=0$ for $l=0,1,\ j=0,1,2$.
Now, starting from initial data $(0,z_2(0),0)$, we have
\begin{align*}
0=\left.\dot z_1\right|_{t=0}=-2\im C^D_{1,2,0}(|z_2(0)|^2   )z_2(0).
\end{align*}
Therefore, we have $C^D_{1,2,0}=0$.
Further, by the relation $C^D_{k,k-l,j}=\overline{C^D_{k,l,j}}$, we have $C^D_{1,2,1}=0$.

Similarly, we have $C^D_{1,1,0}=C^D_{1,1,1}=0$.
\end{proof}

\subsection{Birkhoff normal form}

We now go in to the Birkhoff normal form argument, which make us able to erase ``nonresonant" terms in the expansion of the energy.
We set

\begin{align}
R(k)&:=\begin{cases} \emptyset,& \mathrm{if}\ k:\mathrm{odd},\\
\{k/2\},\ & \mathrm{if}\ k:\mathrm{even},
\end{cases} \label{28}\\
R(j,k)&:=\{0\leq l \leq k\ |\ e_j + (e_2-e_1)(k-2l)\in (0,4)\}.\label{29}
\end{align}

\begin{proposition}\label{prop:birk}
For arbitrary $M \geq 2$, there exist $a_M, \delta_M>0$ s.t. there exist $\tilde{\mathcal Y}^M\in \mathcal R_{l_e^{a_M}}(a_M, \delta_M)$ s.t. $\mathcal Y^M:=\mathrm{Id}+\tilde{\mathcal Y}^M$ is a canonical change of coordinate (i.e.\ $\(\mathcal Y^M\)^* \Omega_0=\Omega_0$) satisfying
\begin{align}\label{30}
|z^M - z| + \|\eta^M-\eta\|_{l_e^{a_M}} \lesssim |z|^5 \(\|\eta\|_{l_e^{-a_M}}+|z_1z_2|\),
\end{align}
where $z^M=\(\mathcal Y^M\)^* z^D$ and $\eta^M=\(\mathcal Y^M\)^*\eta^D$.
Further, $E^D\circ \mathcal Y^M=E^M(z,\eta   )$ can be expanded as
\begin{align}\label{31}
&E^M(z,\eta   )=E_1(|z_1|^2   )+E_2(|z_2|^2   )+E^0(0,\eta   )\\&
+\sum_{M\geq k\geq 2} \sum_{j=0,1,2}\sum_{l\in  R(k)}C_{k,j,l}^M(|z_j|^2) Z^{l}\bar Z^{k-l}+\sum_{M-1\geq k\geq 1}\sum_{j=1,2}\sum_{l\in R(j,l)}\<Z^{l}\bar Z^{k-l}z_jG_{k,j,l}^M(|z_j|^2), \eta\>+\mathcal R_{M},\nonumber
\end{align}
where $C_{k,j,l}^M\in C^\omega(B_{\R}(\delta^2_M) ;\C)$, $G_{k,j,l}^M\in C^\omega(B_{\R}(\delta^2_M) ;l_e^{a_M})$, $\mathcal R_{M} \in \mathcal R_\R(a_M,\delta_M)$ and $$|\mathcal R_{M}| \lesssim |z_1z_2|^{M+1}+|z|\(|z|+\|\eta\|_{l_e^{-a}}\)\|\eta\|_{l_e^{-a_M}}^2.$$
Further, since $E$ is real valued, we have $\overline{C_{k,j,l}}^M=C_{k,j,k-l}^M$.
\end{proposition}

\begin{remark}
By the definition of $R(k)$, we have
\begin{align*}
\sum_{M\geq k\geq 2} \sum_{j=0,1,2}\sum_{l\in  R(k)}C_{k,j,l}^M(|z_j|^2) Z^{l}\bar Z^{k-l}=\sum_{M\geq 2k'\geq 2} \sum_{j=0,1,2}C_{2k',j,k'}^M(|z_j|^2) |z_1|^{2k'}|z_2|^{2k'}.
\end{align*}
\end{remark}

We will give the proof of Proposition \ref{prop:birk} in section \ref{sec:proofbirk}.

\section{Dispersion}\label{sec:dispersion}

\subsection{Linear estimates}
We will now introduce linear estimates of $e^{-\im t H}$.
Lemmas \ref{lem:l0}--\ref{lem:l3} can be found in \cite{CT09SIAM}.
See also \cite{PS08JMP} and \cite{KPS09SIAM}.
Further, for the recent refinement, see \cite{EKT15JST}. 

In the following we always assume $H$ is generic in the sense of Lemma 5.3 of \cite{CT09SIAM}.
We will prove lemma \ref{lem:l4} in section \ref{sec:proofl4}.

\begin{definition}
For an interval $I\subset \R$, we set
\begin{align*}
\stz(I):=L^6(I;l^\infty)\cap L^\infty(I;l^2),\quad \stz^*(I):=L^{6/5}(I;l^1)+L^1(I;l^2),
\end{align*}
where
\begin{align*}
\|u\|_{L^p l^q(I)}:=\| \|u\|_{l^q} \|_{L^p}=\(\int_I \(\sum_{n\in \Z}|u(t,n)|^q\)^{p/q}\,dt\)^{1/p},
\end{align*}
and
\begin{align*}
\|u\|_{X\cap Y}:=\max\(\|u\|_X, \|u\|_{Y}\),\quad \|u\|_{X+Y}:=\inf_{\substack{u_1+u_2=u\\ u_1\in X, u_2\in Y}}\(\|u_1\|_X+\|u_2\|_Y\).
\end{align*}

\end{definition}

\begin{lemma}[Dispersive estimate]\label{lem:l0}
We have
\begin{align*}
\|e^{-\im t H}P_cf\|_{l^\infty}\lesssim \<t\>^{-1/3}\|f\|_{l^1}.
\end{align*}
\end{lemma}

\begin{lemma}[Strichartz estimate]\label{lem:l1}
Let $I\subset \R$ be an interval.
Then, we have
\begin{align*}
\|e^{-\im t H}P_c f\|_{\stz(I)}\lesssim \|f\|_{l^2},
\end{align*}
and
\begin{align*}
\left\|\int_0^t e^{-\im (t-s)H}P_cg(s)\,ds\right\|_{\stz(I)}\lesssim \|g\|_{\stz^*(I)}.
\end{align*}
\end{lemma}

\begin{lemma}[Kato Smoothing]\label{lem:l2}
Let $I\subset \R$ be an interval.
Let $\sigma>1$.
Then, we have
\begin{align*}
\|e^{-\im t H}P_c f\|_{L^2l^{2,-\sigma}(I)}\lesssim \|f\|_{l^2},
\end{align*}
and
\begin{align*}
\left\| \int_0^te^{-\im(t-s)H}P_c g(s)\,ds\right\|_{L^2l^{2,-\sigma}(I)}\lesssim\|g\|_{L^2l^{2,\sigma}(I)}.
\end{align*}
\end{lemma}

\begin{lemma}\label{lem:l3}
Let $I\subset \R$ be an interval.
Let $\sigma>1$.
Then, we have
\begin{align*}
\left\|\int_0^t e^{-\im(t-s)H}P_cg(s,\cdot)\,ds\right\|_{\stz (I)}\lesssim \|g\|_{L^2l^{2,\sigma}(I)}.
\end{align*}
\end{lemma}

\begin{remark}In \cite{CT09SIAM}, the estimates of Lemmas \ref{lem:l2} and \ref{lem:l3} are expressed in the time averaging norm $l^{2,\pm \sigma}_n L^2_t$.
However, since both time and space are $L^2$ ($l^2$) based norms, we can exchange them by Fubini.
\end{remark}

\begin{lemma}\label{lem:l4}
Let $I\subset \R$ be an interval.
Let $\sigma>7/2$.
Then, for $t\geq 0$, we have
\begin{align*}
\left\|e^{-\im tH}R_H^+(\omega_*)P_cf \right\|_{l^{2,-\sigma}}\lesssim \<t\>^{-3/2}\|f\|_{ l^{2,\sigma}}.
\end{align*}
\end{lemma}

\subsection{Bootstrapping}

We now solve DNLS \eqref{1}.
We use the normal form Proposition \ref{prop:birk} with $M=2N_0$.
In this case, $E^{2N_0}$ can be written as
\begin{align}\label{35}
E^{2N_0}(z,\eta   )=&E_1(|z_1|^2   )+E_2(|z_2|^2   )+E(\eta   )
+A(|z_1|^2,|z_2|^2   )\\&+\<\bar  z_1^{N_0-1}z_2^{N_0}G_{N_0-1,2,0}^{2N_0}(0   ), \eta\>+\tilde {\mathcal R}_{2N_0},\nonumber
\end{align}
where
\begin{align*}
\tilde {\mathcal R}_{2N_0}=&\<\bar  z_1^{N_0-1}z_2^{N_0}G_{N_0-1,2,0}^{2N_0}(|z_2|^2   )-G_{N_0-1,2,0}^{2N_0}(0   ), \eta\>\\&+\sum_{2N_0-1\geq k\geq N_0}\sum_{j=1,2}\sum_{l\in R(j,k)}\<Z^{l}\bar Z^{k-l}z_jG_{k,j,l}^{2N_0}(|z_j|^2   ), \eta\>+\mathcal R_{2N_0},\\
A(|z_1|^2,|z_2|   )=&\sum_{2N_0\geq k\geq 2} \sum_{j=1,2}\sum_{l\in  R(k)}C_{k,l,j}^{2N_0}(|z_j|^2   ) Z^{l}\bar Z^{k-l},
\end{align*}
and $|A(|z_1|^2,|z_2|^2   )|\lesssim |z_1|^2|z_2|^2$ and $|\tilde {\mathcal R}_{2N_0}|\lesssim |z|^2\(|z_1^{N_0-1}z_2^{N_0}|^2+\|\eta\|_{l_e^{-a_{2N_0}}}^2\)+|z|\|\eta\|_{l_e^{-a_{2N_0}}}^3$. 

\begin{remark}
Notice that by the definition of $N_0$ and $R(j,k)$ (see \eqref{2.1} and remark \ref{rem:0}), we have $R(j,k)=\emptyset$ if $j=1$ and $k\leq N_0-1$ or $j=2$ and $k\leq N_0-2$.
Further, $R(2,N_0-1)=\{0\}$.
Consequently, we have $$\sum_{N_0-1\geq k\geq 1}\sum_{j=1,2}\sum_{l\in R(j,l)}\<Z^{l}\bar Z^{k-l}z_jG_{k,j,l}^{2N_0}(|z_j|^2), \eta\>=\<\bar Z^{N_0-1}z_2G_{N_0-1,2,0}^{2N_0}(|z_2|^2), \eta\>.$$
\end{remark}
Therefore, by \eqref{25} and \eqref{26}, we have

\begin{align}
\im \dot z_1 &= e_1 z_1 + A_1(|z_1|^2,|z_2|^2   )z_1+(N_0-1)\bar z_1^{N_0-2}z_2^{N_0}(G  , \eta)+\mathcal R_1,\label{41}\\
\im \dot z_2 &= e_2 z_2 + A_2(|z_1|^2,|z_2|^2   )z_2+N_0z_1^{N_0-1}\bar z_2^{N_0-1}(\bar G  , \bar \eta)+\mathcal R_2,\label{42}\\
\im \eta_t &= H\eta + P_c \beta(|\eta|^2   )\eta+ \bar z_1^{N_0-1}z_2^{N_0}G   +  \mathcal R_{\eta},\label{43}
\end{align}
where $G  =G_{N_0-1,2,0}^{2N_0}(0   )$ and 
\begin{align*}
&A_j(|z_1|^2,|z_2|^2)z_j :=\(e_j(|z_j|^2   )-e_j\)z_j+2(1+\tilde\gamma_j(|z_j|^2   ))\partial_{\bar j} A(|z_1|^2,|z_2|^2   ),\\
&\mathcal R_1 := (N_0-1)\tilde \gamma_1(|z_1|^2   )\bar z_1^{N_0-2}z_2^{N_0}(G, \eta)+2(1+\tilde \gamma_1(|z_1|^2   ))\partial_{\bar 1}\tilde {\mathcal R}_{2N_0},\\
&\mathcal R_2 := N_0\tilde \gamma_2(|z_2|^2   )\bar z_1^{N_0-1}z_2^{N_0-1}\overline{(G, \eta)}+2(1+\tilde \gamma_2(|z_2|^2   ))\partial_{\bar 2}\tilde {\mathcal R}_{2N_0},\\&
\mathcal R_\eta :=\nabla_\eta \tilde{\mathcal R}_{2N_0}.
\end{align*}
Notice that we have used
\begin{align*}
2(1+\tilde\gamma_j(|z_j|^2   ))\partial_{\bar j} E_j(|z_j|^2   )=e_j(|z_j|^2   )z_j.
\end{align*}

$\mathcal R_X$ ($X=1,2,\eta$) satisfies
\begin{align*}
&|\mathcal R_j|\lesssim |z|\(|z_1^{N_0-1}z_2^{N_0}|^2+\|\eta\|_{l_e^{-a}}^2\)+\|\eta\|_{l_e^{-a}}^3,\\&
\|\mathcal R_\eta\|_{l_e^a}\lesssim |z|^2|z_1^{N_0-1}z_2^{N_0}|+ |z|^2\|\eta\|_{l_e^{-a}}+\|\eta\|_{l_e^{-a}}^2.
\end{align*}

First, notice that by the mass conservation, we have the following.
\begin{proposition}\label{prop:l2con}
There exist $\varepsilon_0$ s.t. for $0<\varepsilon<\varepsilon_0$, if $|z(0)|+\|\eta(0)\|_{l^2}\leq \varepsilon$, we have
\begin{align*}
\|\eta\|_{L^\infty l^2}\lesssim \varepsilon,\quad
\|z\|_{L^\infty}\lesssim \varepsilon.
\end{align*}

\end{proposition}

Further, by \eqref{41} and \eqref{42} and Proposition \ref{prop:l2con}, we have
\begin{align}\label{43.0}
|\im \dot z_j - e_j z_j|\lesssim |z|^3+\varepsilon^{2}\|\eta\|_{l^{2}}\lesssim \varepsilon^3.
\end{align}

The main estimate is the following.
\begin{proposition}\label{prop:boot}
Set $\Gamma  :=-\mathrm {Im}(G, R_H^+(\omega_*)G)$ and assume $\Gamma  >0$.
Then, there exist $\varepsilon_0$ and $C>0$ s.t. for $0<\varepsilon<\varepsilon_0$, if $|z(0)|+\|\eta(0)\|_{l^2}\leq \varepsilon$, we have
\begin{align*}
&\|\eta\|_{\stz\cap L^2l^{2,-\sigma}}\leq C \varepsilon,\quad
\|z_1^{N_0-1}z_2^{N_0}\|_{L^2}\leq C \varepsilon.
\end{align*}
\end{proposition}

In the following, we prove Proposition \ref{prop:boot} by assuming that the estimates hold in a time interval $[0,T]$ with the bound $C=C_0 $ for sufficiently large $C_0$ (but $C_0 \varepsilon\ll1)$.
Then, we show that we can improve the bound to $\frac 1 2 C_0 $.
The key point is that there is an transfer of energy from the ODE part to the PDE part.
This can be seen by the integrability of $|z_1^{N_0-1}z_2^{N_0}|^2$, which implies either one of  $z_1$ or $z_2$ must decay.

\begin{lemma}\label{lem:b1}
Under the above assumption, we have
\begin{align*}
\|\eta\|_{\stz\cap L^2l^{2,-\sigma}(0,T)}\lesssim \|\eta(0)\|_{l^2} + \| z_1^{N_0-1}z_2^{N_0}\|_{L^2(0,T)}.
\end{align*}
\end{lemma}

\begin{proof}
By Duhamel formula, we have
\begin{align*}
\eta(t)=&e^{-\im t H}\eta(0)-\im \int_0^t e^{-\im (t-s)H}P_c\beta(|\eta(s)|^2   )\,ds\\&-\im \int_0^t e^{-\im (t-s)H}\bar z_1^{N_0-1}(s)z_2^{N_0}(s)G  \,ds-\im \int_0^t e^{-\im (t-s)H}\mathcal R_\eta\,ds.
\end{align*}
By Lemmas \ref{lem:l1} and \ref{lem:l2}, we have
\begin{align*}
\|e^{-\im t H}\eta(0)\|_{\stz\cap L^2l^{2,-\sigma}(0,T)}\lesssim \|\eta(0)\|_{l^2}.
\end{align*}
Next, by Lemmas \ref{lem:l2} and \ref{lem:l3}, we have
\begin{align}
&\|\int_0^t e^{-\im (t-s)H}\bar z_1^{N_0-1}(s)z_2^{N_0}(s)G  \,ds\|_{\stz\cap L^2l^{2,-\sigma}(0,T)}\lesssim \| z_1^{N_0-1}z_2^{N_0}\|_{L^2(0,T)}\|G\|_{l^{2,\sigma}},\nonumber\\&
\|\int_0^t e^{-\im (t-s)H}\mathcal R_\eta(s)\,ds\|_{\stz\cap L^2l^{2,-\sigma}(0,T)}\lesssim \| \mathcal R_\eta\|_{L^2l^{2,\sigma}(0,T)}\nonumber\\&\quad\quad\lesssim \varepsilon (\| z_1^{N_0-1}z_2^{N_0}\|_{L^2(0,T)}+\|\eta\|_{L^2l^{2,-\sigma}(0,T)})\label{43.1},
\end{align}
Finally, for the second term, we have
\begin{align*}
\|\int_0^t e^{-\im (t-s)H}P_c\beta(|\eta(s)|^2   )\,ds\|_{\stz(0,T)}\lesssim \|\eta^7\|_{L^1l^2(0,T)}\leq \|\eta\|_{\stz(0,T)}^7,
\end{align*}
where we have used the fact $\|\eta\|_{L^7l^{14}}\lesssim \|\eta\|_{\stz}$.
\begin{align}
&\|\int_0^t e^{-\im (t-s)H}P_c\beta(|\eta(s)|^2   )\,ds\|_{L^2l^{2,-\sigma}(0,T)}\lesssim \|\int_0^{T} |e^{-\im (t-s)H}P_c\beta(|\eta(s)|^2   )|\,ds\|_{L^2l^{2,-\sigma}(0,T)}\nonumber\\&\lesssim\int_0^T \|e^{-\im (t-s)H}P_c\beta(|\eta(s)|^2   )\|_{L^2_tl^{2,-\sigma}}\,ds\nonumber\\&\lesssim
\int_0^T \|\eta^7(s)\|_{l^2}\,ds=\|\eta^7\|_{L^1l^2(0,T)}\lesssim \|\eta\|_{\stz(0,T)}^7\label{43.2},
\end{align}
where we have used the Minkowski inequality in the second line.

Combining the above estimates, we have
\begin{align*}
\|\eta\|_{\stz\cap L^2l^{2,-\sigma}(0,T)}\lesssim \|\eta(0)\|_{l^2}+\|z_1^{N_0-1}z_2^{N_0}\|_{L^2(0,T)}+\|\eta\|_{\stz\cap L^2l^{2,-\sigma}(0,T)}^7.
\end{align*}
Therefore, combining with the assumption of proposition, we have the conclusion.
\end{proof}

Now, set $\omega_*:=\omega_{N_0}=e_1+N_0(e_2-e_1)$ and $Y=-\bar z_1^{N_0-1}z_2^{N_0} R_H^+(\omega_*)G$ and set $\eta=Y+g$.
Notice that $Y$ is the solution of \eqref{43} without $P_c\beta(|\eta|^2   )$ and $\mathcal R_\eta$ with the assumption $\im \dot z_j=e_jz_j$.
Thus, $g$ can be considered to be a remainder term.
\begin{lemma}
We have
\begin{align*}
\|g\|_{L^2l^{2,-\sigma}(0,T)}\lesssim \|\eta(0)\|_{l^2}+ C_0\varepsilon^2 .
\end{align*}
\end{lemma}

\begin{proof}
First, since $(H-\omega_*)Y+\bar z_1^{N_0-1}z_2^{N_0} G=0$, we have
\begin{align*}
\im \dot g = Hg + P_c\beta(|\eta|^2   )+\mathcal R_\eta + (\omega_*-\im \partial_t)Y.
\end{align*}
Further, since $\omega_*=N_0 e_2-(N_0-1)e_1$, and
\begin{align*}
\im \partial_t Y &= \im (N_0-1)\dot\bar{z_1} \bar z_1^{N_0-2}z_2^{N_0}R_H^+(\omega_*)G+\im N_0 \dot z_2\bar z_1^{N_0-1}z_2^{N_0-1}R_H^+(\omega_*)G\\&=\omega_* Y + (-(N_0-1)\overline{(\im \dot z_1-e_1z_1)}z_2+N_0\bar z_1(\im z_2-e_2z_2))\bar z_1^{N_0-2}z_2^{N_0-1}R_H^+(\omega_*)G.
\end{align*}
Therefore, by Duhamel formula, we have
\begin{align}\label{44}
&\|g\|_{L^2l^{2,-\sigma}(0,T)}\lesssim \|e^{-\im t H} \eta(0)\|_{L^2l^{2,-\sigma}(0,T)} + \|e^{-\im t H} Y(0)\|_{L^2l^{2,-\sigma}(0,T)}\\&\quad+\|\int_0^t e^{-\im (t-s)H}\mathcal R_\eta(s)\,ds\|_{L^2l^{2,-\sigma}(0,T)}\nonumber+\|\int_0^t e^{-\im (t-s)H}P_c\beta(|\eta(s)|^2   )\,ds\|_{L^2l^{2,-\sigma}(0,T)}\\&\quad+\|\int_0^t e^{-\im (t-s)H}(\omega_*-\im \partial_t)Y\,ds\|_{L^2l^{2,-\sigma}(0,T)}.\nonumber
\end{align}
Now, by Strichartz estimate Lemma \ref{lem:l1}, the first term in the r.h.s.\ of \eqref{44} can be bounded by $\|\eta(0)\|_{l^2}$.
Next, by the definition of $Y$ and Lemma \ref{lem:l4}, we have
\begin{align*}
\|e^{-\im t H}Y(0)\|_{L^2 l^{2,-\sigma}(0,T)}&\lesssim \|z_1(0)^{N_0-1}z_2(0)^{N_0}| \|e^{-\im t H}R_H^+(\omega_*)G\|_{L^2l^{2,-\sigma}(0,T)}\\&\lesssim \varepsilon^{2N_0-1}\|\<t\>^{-4/3}\|_{L^2(0,T)}\|G\|_{l^{2,\sigma}}\lesssim \varepsilon^2.
\end{align*}
The third term and fourth term in the r.h.s.\ of \eqref{44} are already estimated in \eqref{43.1} and \eqref{43.2} and we have
\begin{align*}
\|\int_0^t e^{-\im (t-s)H}\mathcal R_\eta(s)\,ds\|_{L^2l^{2,-\sigma}(0,T)}\nonumber+\|\int_0^t e^{-\im (t-s)H}P_c\beta(|\eta(s)|^2   )\,ds\|_{L^2l^{2,-\sigma}(0,T)}\lesssim C_0 \varepsilon^2.
\end{align*}
Finally, for the last term in the r.h.s.\ of \eqref{44}, we have
\begin{align*}
&\|\int_0^t e^{-\im (t-s)H}(\omega_*-\im \partial_t)Y\,ds\|_{L^2l^{2,-\sigma}(0,T)}\\&\lesssim \|\int_0^T \|e^{-\im (t-s)H}\overline{(\im \dot z_1-e_1z_1)}\bar z_1^{N_0-2}z_2^{N_0}R_H^+(\omega_*)G\|_{l^{2,-\sigma}}\,ds\|_{L^2}\\&\quad
+ \|\int_0^T \|e^{-\im (t-s)H}{(\im \dot z_2-e_2z_2)}\bar z_1^{N_0-1}z_2^{N_0-1}R_H^+(\omega_*)G\|_{l^{2,-\sigma}}\,ds\|_{L^2}\\&
\lesssim \varepsilon^2 \| \int_0^T \<t-s\>^{-3/2}\(|z_1^{N_0-1}z_2^{N_0}(s)|+\|\eta\|_{l^{2,-\sigma}}\)\|G\|_{l^{2,\sigma}}\,ds\|_{L^2}\\&\lesssim C_0\varepsilon^3. 
\end{align*}
Therefore, we have the conclusion.
\end{proof}

Now, substituting $\eta=Y+g$ into the equation, we have
\begin{align}
&\frac 1 2 \frac{d}{dt}|z_1|^2 = -(N_0-1)|z_1|^{2(N_0-1)}|z_2|^{2N_0}\mathrm{Im}(G, R_H^+(\omega_*)G)+(N_0-1)\mathrm{Im}\bar z_1^{N_0-1}z_2^{N_0}(G,g)+\mathrm{Im}\mathcal R_1 \bar z_1,\label{45}\\&
\frac 1 2 \frac{d}{dt}|z_2|^2 = N_0|z_1|^{2(N_0-1)}|z_2|^{2N_0}\mathrm{Im}( G, R_H^+(\omega_*)G)+N_0\mathrm{Im} z_1^{N_0-1}\bar z_2^{N_0}(\bar G,\bar g)+\mathrm{Im}\mathcal R_2 \bar z_2,\label{46}
\end{align}

Recall that we have assumed
\begin{align*}
\Gamma  :=-\mathrm {Im}(G, R_H^+(\omega_*)G)=-\mathrm{Im}(G, \im \pi \delta(H-\omega_*)G)=\pi\<G, \delta(H-\omega_*)G\>>0.
\end{align*}

\begin{remark}
Notice that since $G  $ is analytic w.r.t.\ $\lambda$, $\Gamma$ is also analytic w.r.t.\ $\lambda$.
\end{remark}

Now, integrating the second equation on time interval $[0,T]$, we have
\begin{align*}
|z_2(T)|^2+\Gamma  \|z_1^{N_0-1}z_2^{N_0}\|_{L^2(0,T)}^2&\lesssim |z_2(0)|^2+\|z_1^{N_0-1}z_2^{N_0}\|_{L^2}\|G\|_{l^{2,\sigma}}\|g\|_{L^2l^{2,-\sigma}}+\|\mathcal R_2 \bar z_2\|_{L^2}\\&\lesssim \varepsilon^2 +C_0 \varepsilon (\varepsilon+ C_0 \varepsilon^2)+ C_0^2 \varepsilon^3.
\end{align*}
Therefore, we have
\begin{align}\label{46.1}
\|z_1^{N_0-1}z_2^{N_0}\|_{L^2}^2\lesssim  (1+C_0^2 \varepsilon)\varepsilon^2,
\end{align}
which gives us the conclusion of Proposition \ref{prop:boot}.

Finally, we show that $|z_j(t)|$ have to converge and one of the limit must be $0$.
\begin{proposition}
Under the assumption of Proposition \ref{prop:boot}, there exists $\rho_j\geq 0$ with $\rho_1\rho_2=0$ such that
$
|z_j(t)|\to \rho_j,\text{ as }t\to \infty.
$
\end{proposition}
\begin{proof}
First, to show $|z_j|$ converge to some $\rho_j$, it suffices to show $\frac{d}{dt} |z_j|^2 \in L^1([0,\infty))$.
However, this follows immediately from Proposition \ref{prop:boot} and \eqref{45}, \eqref{46}. 
Next, if $\rho_1\rho_2\neq 0$, this will contradict with the fact that $|z_1^{2(N_0-1)}z_2^{2N_0}|$ is integrable.
Therefore, we have the conclusion.
\end{proof}

\section{Proof of main theorems}\label{sec:proofmain}
Because of Proposition \ref{prop:boot}, we will get Theorem \ref{thm:1} and Theorem \ref{thm:3} immediately.
Further, Theorem \ref{thm:2} will be an direct consequence of Theorem \ref{thm:1} with a simple observation.

\begin{proof}[Proof of Theorem \ref{thm:1}]
First, because $\|\eta\|_{\stz(0,\infty)}<\infty$, there exists $\eta_+\in l^2$ s.t.\ $\|\eta(t)-e^{\im t \Delta}\eta_+\|_{l^2}\to 0$ as $t\to 0$.
Next, by \eqref{45} and \eqref{46}, we see that $|z_j|$ converges.
Further, since $|z_1^{N_0-1}z_2^{N_0}|$ is integrable, one of $j=1,2$ has to converge to 0.

Finally, since the original coordinate and the new coordinate which we used above, is connected by the relation \eqref{30}, we can translate the result for the new coordinate to the original coordinate.
\end{proof}

\begin{proof}[Proof of Theorem \ref{thm:3}]
We have
\begin{align*}
\frac 1 2 \frac{d}{dt}(N_0|z_1|^2+(N_0-1)|z_2|^2)=N_0\Im \mathcal R_1 \bar z_1 +(N_0-1)\mathcal R_2 \bar z_2,
\end{align*}
Therefore, $N_0|z_1|^2+(N_0-1)|z_1|^2$ almost conserves and further, if $|z_1(t)|^2\to \rho_+^2$, we have
\begin{align*}
|\rho_+^2-\(|z_1(0)|^2+(1-N_0^{-1})|z_2(0)|^2\)|\lesssim \varepsilon^4
\end{align*}
and if we have $|z_2(t)|^2 \to \rho_+^2$, we have
\begin{align*}
| \rho_+^2 -\(\frac{N_0}{N_0-1}|z_1(0)|^2+|z_2(0)|^2\)|\lesssim \varepsilon^4.
\end{align*}
So, again translating the new coordinate to the old coordinate, we have the conclusion.
Note that from Proposition \ref{prop:1}, Lemma \ref{lem:2} and Lemma \ref{lem:3}, we have
$|z_j(0)-\(u_0,\phi_j\)|\lesssim \varepsilon^7$ so we can replace $z_j(0)$ by $|(u_0,\phi_j)|$ in the conclusion of the Theorem.
\end{proof}

For the proof of Theorem \ref{thm:2} is completely the same as the proof of Theorem 1.4 of \cite{CuMaAPDE}.
Therefore, we omit the proof.
See also \cite{CM16JNS}.

\section{Proof of technical propositions}\label{sec:prooftech}

\subsection{Proof of Darboux theorem (Proposition \ref{prop:dar})}\label{sec:proofdar}

In this section, we prove Darboux theorem, which is a change of coordinate to make the original coordinate to be a ``canonical" coordinate.
For the discussion of the strategy of the proof, see \cite{MDNLS1}.
We set
\begin{align*}
&B(u)X= \frac 1 2 \Omega(u,X),\\&
B_0(u)X=\frac{1}{2} \(\sum_{j=1,2}\Omega(\phi_j(z_j),d\phi_j(z_j)X)+\Omega(\eta, d\eta X)\).
\end{align*}

\begin{lemma}\label{lem:prd1}
Let $\delta>0$ sufficiently small.
Then, there exists $F_\eta\in \mathcal R_{l_e^a}(a,\delta)$ and $F_{j,A}\in \mathcal R_{\R}(a,\delta)$  (recall Definition \ref{def:remainder}) s.t. for some $C\in \mathcal R_\R(a,\delta)$, we have
\begin{align*}
B(u)-B_0(u)-dC=\sum_{j=1,2,A=R,I}F_{j,A}dz_{j,A}+\<F_{\eta},d\eta\>=:\Gamma.
\end{align*}
Further, we have
\begin{align}\label{48}
\|F_\eta\|_{l_e^a}+\sum_{j=1,2,A=R,I}|F_{j,A}|\lesssim |z|^5 |z_1z_2|+|z|^6\|\eta\|_{l_e^{-a}}.
\end{align}
\end{lemma}

\begin{proof}
First, since
\begin{align*}
B(u)=\frac 1 2 \Omega(\phi_1(z_1)+\phi_2(z_2)+\eta + \<\alpha_{k,B},\eta\>\phi_{k,B},d\phi_1(z_1)+d\phi_2(z_2)+d\eta+d(\<\alpha_{j,A},\eta\>)\phi_{j,A}),
\end{align*}
we have
\begin{align*}
2B(u)-2B_0(u)=&\Omega(\phi_1(z_1),d\phi_2(z_2)+d\eta + d\<\alpha_{j,A},\eta\>\phi_{j,A})\\&+
\Omega(\phi_2(z_2),d\phi_1(z_1)+d\eta + d\<\alpha_{j,A},\eta\>\phi_{j,A})\\&
+
\Omega(\eta,d\phi_1(z_1)+d\phi_2(z_2) + d\<\alpha_{j,A},\eta\>\phi_{j,A})\\&
+
\Omega(\<\alpha_{k,B},\eta\>\phi_{k,B},d\phi_1(z_1)+d\phi_2(z_2) +d\eta+ d\<\alpha_{j,A},\eta\>\phi_{j,A}).
\end{align*}
Now, notice that we have
\begin{align*}
\Omega(\phi_1(z_1),d\phi_2(z_2))=&d \Omega(\phi_1(z_1),\phi_2(z_2))+ \Omega(\phi_2(z_2),d\phi_1(z_1)),\\
\Omega(\phi_1(z_1)+\phi_2(z_2),d\eta)=&d \Omega(\phi_1(z_2)+\phi_2(z_2),\eta)+\Omega(\eta,d\phi_1(z_1)+d\phi_2(z_2)),\\
\Omega(\phi_1(z_1)+\phi_2(z_2)+\eta,d\<\alpha_{j,A},\eta\>\phi_{j,A})=&d\Omega(\phi_1(z_1)+\phi_2(z_2)+\eta,\<\alpha_{j,A},\eta\>\phi_{j,A})\\ &\quad+\<\alpha_{j,A},\eta\> \Omega(\phi_{j,A}, d\phi_1(z_1)+d\phi_2(z_2)+d\eta),
\end{align*}
and
\begin{align*}
\Omega(\phi_2(z),d\phi_1(z_1))&=\Omega(z_2\phi_2+q_2(z_2), \phi_1 dz_1 + d q_1(z_1))\\&
=\Omega(z_2\phi_2, d q_1(z_1))+\Omega(q_2(z_2),\phi_1 dz_1 + dq_1(z_1))\\&
=d \Omega(q_2(z_2),z_1\phi_1)+\Omega(z_1\phi_1, d q_2(z_2))+ \Omega(z_2\phi_2, d q_1(z_1))+ \Omega(q_2(z_2), d q_1(z_1)).
\end{align*}
Therefore, for
\begin{align*}
2C=& \Omega(\phi_1(z_1),\phi_2(z_2))+\Omega(\phi_1(z_2)+\phi_2(z_2),\eta)\\&+\Omega(\phi_1(z_1)+\phi_2(z_2)+\eta,\<\alpha_{j,A},\eta\>\phi_{j,A})+\Omega(q_2(z_2),z_1\phi_1),
\end{align*}
we have
\begin{align}
B(u)-B_0(u)-dC =& \Omega(\eta,d\phi_1(z_1)+d\phi_2(z_2))\label{49}\\&+\<\alpha_{j,A},\eta\> \Omega(\phi_{j,A}, d\phi_1(z_1)+d\phi_2(z_2)+d\eta)\nonumber\\&
+\frac 1 2 \Omega(\<\alpha_{k,B},\eta\>\phi_{k,B},d\<\alpha_{j,A},\eta\>\phi_{j,A})\nonumber\\&
+\frac 1 2 \(\Omega(z_1\phi_1, d q_2(z_2))+ \Omega(z_2\phi_2, d q_1(z_1))+ \Omega(q_2(z_2), d q_1(z_1))\).\nonumber
\end{align}
Setting $\Gamma=\mathrm{r.h.s.\ of}\ \eqref{49}$, we see that $F_\eta \in \mathcal R_{l_e^{a}}(a,\delta)$, $F_{j,A}\in \mathcal R_\R(a,\delta)$ and further \eqref{48} is satisfied.
\end{proof}

We set
\begin{align*}
\Omega_s=\Omega_0 + s(\Omega-\Omega_0).
\end{align*}
\begin{lemma}
Let $\delta>0$ sufficiently small.
Then, there exists $\mathcal X_\eta(z,\eta,\lambda,s)$, $\mathcal X_{j,A}(z,\eta,\lambda,s)$ s.t.
\begin{align}
&\mathcal X_\eta \in C^\omega(B_{\C^2\times l_{e,c}^{-a}}(0,\delta)\times B_{\R}(0,1)\times B_{\R}(0,2);l_{e,c}^a),\label{49.1}\\&\mathcal X_{j,A}\in C^\omega(B_{\C^2\times l_{e,c}^{-a}}(0,\delta)\times B_{\R}(0,1)\times B_{\R}(0,2);\R), \label{49.2}
\end{align}
s.t. $\mathcal X^s:=\sum_{j=1,2,A=R,I}\mathcal X_{j,A}(\cdot,\cdot,\cdot,s)\partial_{z_{j,A}}+\mathcal X_\eta (\cdot,\cdot,\cdot,s) \nabla_\eta$ satisfies $i_{\mathcal X^s} \Omega_s = -\Gamma$.
Further, we have
\begin{align}\label{49.3}
\|\mathcal X_\eta\|_{l_e^a}+\sum_{j=1,2,A=R,I}|\mathcal X _{j,A}|\lesssim |z|^5 |z_1z_2|+|z|^6\|\eta\|_{l_e^{-a}}.
\end{align}
\end{lemma}

\begin{proof}
We directly solve
\begin{align*}
\Omega_0(\mathcal X^s,\cdot)+s\(\Omega(\mathcal X^s,\cdot)-\Omega_0(\mathcal X^s,\cdot)\)=-\Gamma.
\end{align*}
First,
\begin{align*}
\Omega_0(\mathcal X^s,Y)=\Omega(D_{j,B}\phi_j(z_j),D_{j,A}\phi_{j}(z_j)) \mathcal X_{j,B} Y_{j,A}+ \Omega(\mathcal X_\eta,Y_\eta).
\end{align*}
Next, since $\Omega-\Omega_0=d \Gamma$, we have
\begin{align*}
&\Omega(\mathcal X^s,Y)-\Omega_0(\mathcal X^s,Y)=\(D_{k,B}F_{j,A}-D_{j,A}F_{k,B}\) \mathcal X_{k,B}Y_{j,A} + \<\nabla_\eta F_{j,A}, \mathcal X_\eta\>Y_{j,A}-\<\nabla_\eta F_{k,B}, Y_\eta\>\mathcal X_{k,B}\\&\quad
+\<D_{k,B}F_\eta,Y_\eta\>\mathcal X_{k,B}-\<D_{j,A}F_\eta,\mathcal X_\eta\>Y_{j,A}+\<d_\eta F_\eta(\mathcal X_\eta),Y_\eta\>-\<d_\eta F_\eta (Y_\eta),\mathcal X_\eta\>
\end{align*}
Therefore, we have
\begin{align}
&\im \mathcal X_\eta + s\(-\mathcal X_{k,B}\nabla_\eta F_{k,B}+\mathcal X_{k,B}D_{k,B}F_\eta + d_\eta F_\eta(X_\eta)-(d_\eta F_\eta)^* \mathcal X_\eta\)=-F_\eta,\label{49.4}\\&
\Omega(D_{j,B}\phi_j(z_j),D_{j,A}\phi_j(z_j))\mathcal X_{j,B}+s\((D_{k,B}F_{j,A}-D_{j,A}F_{k,B}) \mathcal X_{k,B}+\<\nabla_\eta F_{j,A}, \mathcal X_\eta\>-\<D_{j,A}F_\eta,\mathcal X_\eta\>\)\nonumber\\&=-F_{j,A}.\label{49.5}
\end{align}
First, for fixed $\mathcal X_{k,B}$, we can solve \eqref{49.4} by Neumann series.
Notice that the solution $\mathcal X_\eta$ becomes analytic w.r.t.\ $z,\eta,\lambda,s$ and $\mathcal X_{k,B}$.
Next, since $\Omega(D_{j,B}\phi_j(z_j),D_{j,A}\phi_j(z_j))$ is invertible, we can solve \eqref{49.5} again by Neumann series.
Therefore, we obtain $\mathcal X_{j,A}$ and $\mathcal X_\eta$ which satisfies \eqref{49.1}, \eqref{49.2} and \eqref{49.3}.
\end{proof}

We now consider the following system
\begin{align}
&\frac{\partial}{\partial s} r_z (z,\eta,\lambda,s)=\mathcal X_z(z+r_z,\eta+r_\eta,\lambda,s),\label{50}\\&
\frac{\partial }{\partial s} r_\eta (z,\eta,\lambda,s)=\mathcal X_\eta(z+r_z,\eta+r_\eta,\lambda,s),\label{51}
\end{align}
with the initial condition $(r_z,r_\eta)=(0,0)$, where $$\mathcal X_z=(\mathcal X_{1,R}+\im \mathcal X_{1,I},\mathcal X_{2,R}+\im \mathcal X_{2,I})\in C^\omega(B_{\C^2\times l_{e,c}^{-a}}(0,\delta)\times B_{\R}(0,1)\times B_{\R}(0,2);\C^2).$$

\begin{lemma}
Let $\delta>0$ sufficiently small.
Then, there exists $$(r_z,r_\eta)\in C^\omega(B_{\C^2\times P_cl_e^{-a}}(0,\delta)  ;C([0,1];\C^2\times l_{e,c}^a)),$$ s.t.
$(r_z(z,\eta,\lambda,\cdot),r_\eta(z,\eta,\lambda,\cdot))$ is the solution of system \eqref{50}--\eqref{51} and
\begin{align*}
&|r_{z}(z,\eta,\lambda,1)|+
\|r_\eta(z,\eta,\lambda,1)\|_{l_e^{a}}\lesssim |z|^{5}|z_1z_2|+|z|^{6}\|\eta\|_{l_e^{-a}}.
\end{align*}

\end{lemma}

\begin{proof}
We solve \eqref{50}-\eqref{51} by implicit function theorem.

First, for $(w,\xi)\in C([0,1];B_{\C^2\times l_{e,c}^{-a}}(0,\delta))$, set
\begin{align*}
\Phi(z,\eta,\lambda,w,\xi)(s):=(\Phi_z(z,\eta,\lambda,w,\xi)(s),\Phi_\eta(z,\eta,\lambda,w,\xi)(s)),
\end{align*}
where
\begin{align*}
&\Phi_z(z,\eta,\lambda,w,\xi)(s)=w(s)-\int_0^s \mathcal X_z(z+w(\tau),\eta+\xi(\tau),\lambda,\tau)\,d \tau,\\&
\Phi_\eta(z,\eta,\lambda,w,\xi)(s)=\xi(s)-\int_0^s \mathcal X_\eta(z+w(\tau),\eta+\xi(\tau),\lambda,\tau)\,d \tau.
\end{align*}
Notice that $\Phi\in C^\omega(B_{\C^2\times l_{e,c}^{-a}}(0,\delta)  \times B_{C([0,1];\C^2\times l_{e,c}^{-a})};C([0,1];\C^2\times l_{e,c}^{a}))$.
Then, by implicit function theorem, we can show there exist $(x_1(z,\eta,\lambda)(s),x_2(z,\eta,\lambda)(s),\eta(z,\eta,\lambda)(s))$ which satisfies $\Phi=0$.
\end{proof}

\subsection{Proof of Birkhoff normal form (Proposition \ref{prop:birk})}\label{sec:proofbirk}

We prove Proposition \ref{prop:birk} by induction of $M$.
The proof of Proposition \ref{prop:birk} is similar to the proof of Theorem 5.9 of \cite{CuMaAPDE}.
The aim here is to erase the nonresonant terms in the energy expansion.

Before, going in to the induction argument, we introduce some notations.
Let $\N_0=\{0\}\cup\N$.
For $\mathbf m =(\mu,\nu)=(\mu_1,\mu_2,\nu_1,\nu_2) \in \N_0^2\times \N_0^2$ and $z\in \C^2$, we set
\begin{align*}
\mathbf Z^{\mathbf m}:=\mathbf Z^{\mathbf m}(z):=z^\mu\bar z^\nu:=z_1^{\mu_1}z_2^{\mu_2}\bar z_1^{\nu_1}\bar z_2^{\nu_2},
\end{align*}
and $\bar{\mathbf{m}}=(\nu,\mu)$.
\begin{remark}
We have $\overline{\mathbf Z^{\mathbf m}}=\mathbf Z^{\overline{\mathbf m}}$.
\end{remark}
We further set $\delta_{i,j}$ is the usual Kronecker delta, $|\mu|=\mu_1+\mu_2$ for $\mu=(\mu_1,\mu_2)\in\N_0^2$, $\mathbf e=(e_1,e_2)$ and $\mathbf e\cdot (\mu-\nu)=e_1(\mu_1-\nu_1)+e_2(\mu_2-\nu_2)$.
We redefine the resonant set by
\begin{align*}
\mathbf M(k)&:=\{\mathbf m \in  \N_0^2\times \N_0^2\ |\ |\mu|=|\nu|=k\},\\
\mathbf M(k,j)&:=\{\mathbf m \in  \N_0^2\times \N_0^2\ |\ (\mu_1-\delta_{1j},\mu_2-\delta_{2j},\nu_1,\nu_2)\in \mathbf M(k)\},\\
\mathbf R(k)&:=\{\mathbf m \in \mathbf M(k)\ |\ \mu_1=\nu_1\},\quad \quad \ 
\mathbf R(k,j):=\{\mathbf M(k,j)\ |\ 0<\mathbf e\cdot (\mu-\nu)<4 \},\\
\mathbf {NR}(k)&:=\mathbf M(k)\setminus \mathbf R(k),\quad\quad\quad\quad\quad\quad\ \ 
\mathbf {NR}(k,j):=\mathbf M(k,j)\setminus \mathbf R(k,j).
\end{align*}

\begin{remark}
If $\mathbf m=(\mu,\nu)\in \mathbf R(k)$, we automatically have $\mu_2=\nu_2$ because $|\mu|=|\nu|$.
\end{remark}

\begin{remark}
If $\mathbf m=(\mu,\nu)\in \mathbf{NR}(k)$, we automatically have $\mathbf e\cdot(\mu-\nu)\neq 0$.
This is because if $(\mu,\nu)\in \mathbf{NR}(k)$, then we have $\mu_1-\nu_1=-(\mu_2-\nu_2)\neq 0$.
So, $\mathbf e\cdot(\mu-\nu)=(e_2-e_1)(\mu_2-\nu_2)\neq 0$.
However, this is in some sense special for the two eigenvalue case.
When we have three or more eigenvalues, we need to assume an additional nonresonance condition such as (H3) of \cite{CuMaAPDE}.
\end{remark}

\begin{remark}
If $l\in R(k)$ (where the definition of $R(k)$ is given in \eqref{28}), then there exists a corresponding $\mathbf m \in \mathbf R(k)$ s.t.\ $Z^{k-l}\bar Z^{l}=\mathbf Z^{\mathbf m}$ and vise versa, where $Z$ is defined in \eqref{def:Z}.
Similarly, if $l\in R(k,j)$ (where the definition of $R(k)$ is given in \eqref{29}), then there exists a corresponding $\mathbf m\in \mathbf R(k,j)$ s.t.\ $z_jZ^{k-l}\bar Z^l=\mathbf Z^{\mathbf m}$ and the inverse also holds.
\end{remark}

To prove Proposition \ref{prop:birk}, it suffices to prove the following proposition.

\begin{proposition}\label{prop:Mth}
Let $M\geq 2$ and assume that for $a,\delta>0$, 
 there exist $C_{\mathbf m, 0}^{M-1}\in C^\omega(B_{\R}(0,1);\C)$, $C_{\mathbf m, j}^{M-1}\in C^\omega(B_{\R}(0,\delta^2)\times B_{\R}(0,1);\C)$ and $G_{\mathbf m, j}^{M-1}\in C^\omega(B_{\R}(0,\delta^2)\times B_{\R}(0,1);l_{e,c}^a)$ s.t.
\begin{align}
&E^{M-1}(z,\eta   )=E_1(|z_1|^2   )+E_2(|z_2|^2   )+E(\eta   )\label{50.03}\\&+\sum_{2\leq k\leq M-1}\sum_{j=0,1,2}\sum_{\mathbf m \in  \mathbf R(k)}C_{\mathbf m,j}^{M-1}(|z_j|^2   ) \mathbf Z^{\mathbf m}+\sum_{1\leq k\leq M-2}\sum_{j=1,2}\sum_{\mathbf m \in \mathbf R(k,j)}\<\mathbf Z^{\mathbf m}G_{\mathbf m,j}^{M-1}(|z_j|^2   ),  \eta\>\nonumber\\&+\sum_{k\geq M} \sum_{j=0,1,2} \sum_{\mathbf m \in \mathbf M( k)}C_{\mathbf m,j}^{M-1}(|z_j|^2   )\mathbf Z^{\mathbf m} + \sum_{k\geq M-1}\sum_{j=1,2}\sum_{\mathbf m \in \mathbf M(k,j)}\<\mathbf Z^{\mathbf m}G_{\mathbf m,j}^{M-1}(|z_j|^2   ),\eta\>\nonumber\\&+ \mathcal R^{M-1}(z,\eta   ),\nonumber
\end{align}
where $\mathcal R^{M-1}\in R_{\R}(a,\delta)$ and
\begin{align}\label{50.02}
|\mathcal R^{M-1}|\lesssim |z|\(|z|+\|\eta\|_{l_e^{-a}}\)\|\eta\|_{l_e^{-a}}^2,
\end{align}
$C^{M-1}_{\overline{\mathbf m},j}=\overline{C^{M-1}_{\mathbf m,j}}$ and $C^{M-1}_{\mathbf m,j}(0   )=0$ for $j=1,2$.
Then, there exist $a',\delta'>0$ s.t.\ there exists $\tilde{\mathcal Y}^M\in \mathcal R_{l_e^{a'}}$ s.t.\ $\mathcal Y^M =\mathrm{Id}+\tilde{\mathcal Y}^M$ is a canonical change of coordinate  (i.e.\ $(\mathcal Y^M)^* \Omega_0 = \Omega_0$) and $\tilde{\mathcal Y}^M$ satisfies 
\begin{align}\label{51.0}
|(\tilde {\mathcal Y}^M)^*z| + \|(\tilde{\mathcal Y}^M)^*\eta\|_{l_e^a}\lesssim |z|\(\|\eta\|_{l_e^{-a}}+|z_1z_2|\),
\end{align}
and $E^M:=(\mathcal Y^M)^* E^{M-1}$ has the expansion \eqref{50.03} with $M-1$ replaced to $M$ and $a,\delta$ replaced to $a', \delta'$.
\end{proposition}
\begin{remark}
	
In \eqref{50.03}, $C^{M-1}_{\mathbf m, j}(|z_j|^2   )$ with $j=0$ means $C^{M-1}_{\mathbf m, j}  $.
\end{remark}

\begin{remark}
$\mathcal Y^M$ in Proposition \ref{prop:birk} will correspond to $\mathcal Y^2\circ \cdots \circ \mathcal Y^M$, where the latter $\mathcal Y^k$'s is the one given in Proposition \ref{prop:Mth}.
\end{remark}

We will construct $\mathcal Y^M$ in Proposition \ref{prop:Mth} by Hamiltonian vector flow of some auxiliary Hamiltonian.
Therefore, the task will be to construct the auxiliary Hamiltonian to erase the terms
\begin{align*}
\sum_{j=0,1,2} \sum_{\mathbf m \in \mathbf {NR}( M)}C_{\mathbf m,j}^{M-1}(|z_j|^2   )\mathbf Z^{\mathbf m} + \sum_{j=1,2}\sum_{\mathbf m \in \mathbf {NR}(M-1,j)}\<\mathbf Z^{\mathbf m}G_{\mathbf m,j}^{M-1}(|z_j|^2   ),\eta\>.
\end{align*}

Before getting in the details of the proof, we explain the basic strategy of the proof and the role of the nonresonace condition.
We first explain how to erase 
\begin{align*}
C^{M-1}_{\mathbf m_0,0}\mathbf Z^{\mathbf m_0}+\overline{C^{M-1}_{\mathbf m_0,0}}\mathbf Z^{\overline{\mathbf m_0}}+\sum_{j=1,2} \<\mathbf Z^{\mathbf m_j}G^{M-1}_{\mathbf m_j,j},\eta\>,\quad \mathbf m_0\in \mathbf{NR}(M),\ \mathbf m_j\in \mathbf{NR}(M-1,j),
\end{align*}
with $G^{M-1}_{\mathbf m_j,j}$ not depending of $z_j$.
We set the auxiliary Hamiltonian as
\begin{align*}
\chi=b^{M-1}_{\mathbf m_0,0}\mathbf Z^{\mathbf m_0}+\overline{b^{M-1}_{\mathbf m_0,0}}\mathbf Z^{\overline{\mathbf m_0}}+\sum_{j=1,2} \<\mathbf Z^{\mathbf m_j}B^{M-1}_{\mathbf m_j,j},\eta\>.
\end{align*}
Then, the canonical change of coordinate $(z_1,z_2,\eta)\mapsto (z_1+r_1,z_2+r_2,\eta+r_\eta)$ induced by the Hamilton vector field $X_\chi$ will satisfy
\begin{align*}
 r_k\sim  (X_\chi)_{z_k} =-2\im   \partial_{\bar z_k}\chi=& -2\im\(\nu_{0,k}b_{\mathbf m_0,0}^{M-1}\frac{\mathbf Z^{\mathbf m_0}}{\bar z_k}+\mu_{0,k}\overline{b_{\mathbf m_0,0}^{M-1}}\frac{\mathbf Z^{\overline{\mathbf m_0}}}{\bar z_k}\)\\&\quad\quad-\im \sum_{j=1,2} \(\(\nu_{j,k}\frac{\mathbf Z^{\mathbf m}}{\bar z_k} B_{\mathbf m_j,j},\eta\)+\(\mu_{j,k} \frac{\mathbf Z^{\bar{\mathbf m}}}{\bar z_k} \bar B_{\mathbf m_j,j},\bar \eta\)\)\\
 r_\eta \sim  (X_\chi)_\eta =&-\im \nabla_\eta \chi= -\im\sum_{j=1,2} \mathbf Z^{\mathbf m_j}B_{\mathbf m_j,j}^{M-1},
\end{align*}
where $\mathbf m_j=(\mu_j,\nu_j)=(\mu_{j,1},\mu_{j,2},\nu_{j,1},\nu_{j,2})$.
Substituting this into the quadratic part of the energy, we have
\begin{align*}
&\frac12 \sum_{k=1,2}e_k|z_k+r_k|^2+\frac12\<H(\eta+r_\eta),\eta+r_\eta\>=\frac12 \sum_{k=1,2}e_k|z_k|^2+\frac12\<H\eta,\eta\>\\&+\Re\(2\im \mathbf e\cdot(\mu_{0}-\nu_{0})b_{\mathbf m_0,0}^{M-1}\mathbf Z^{\mathbf m_0}\)-\sum_{j=1,2}\<\mathbf Z^{\mathbf m_j}  \im \(H-\mathbf e\cdot (\mu_{j}-\nu_{j})\)B_{\mathbf m_j,j},\eta\>+h.o.t.
\end{align*}
where $h.o.t.$ are the higher order terms.
Thus, if we set
\begin{align*}
\im \mathbf e\cdot(\mu_{0}-\nu_{0})b_{\mathbf m_0,0}^{M-1}=-C^{M-1}_{\mathbf m_0,0}\text{ and }\im \(H-\mathbf e\cdot (\mu_{j}-\nu_{j})\)B_{\mathbf m_j,j}=-G^{M-1}_{\mathbf m_j,j}, 
\end{align*}
then these terms will cancel with the terms which we wanted to erase.
It is now clear that the nonresonance condition enables us to solve the above equations.

We will now go in to the detail of the proof.
Although the basic strategy is simple as above, the actual proof will be involved because $C_{\mathbf m,j}^{M-1}$ and $G_{\mathbf m,j}^{M-1}$ depends on $z_j$ and we have to erase them at once.
To do so, we will use implicit function theorem.
Also, we will have to estimate the error of the time one mapping of the Hamilton vector flow.

As explained above we will consider the auxiliary Hamiltonian in the form
\begin{align*}
\chi(z,\eta,b,B   )=\sum_{j=0,1,2} \sum_{\mathbf m\in \mathbf{NR}(M)}b_{\mathbf m,j}(|z_j|^2   )\mathbf Z^{\mathbf m} +  \sum_{j=1,2} \sum_{\mathbf m \in \mathbf{NR}(j,M-1)}\<\mathbf Z^{\mathbf m}B_{\mathbf m,j}(|z_j|^2   ),\eta\>,
\end{align*}
where $b_{\mathbf m,0}\in C^\omega(B_\R(0,1);\C)$, $b_{\mathbf m,j}\in C^\omega(B_{\R}(0,\delta^2)  ;\C)$ and $B_{\mathbf m,j}\in C^\omega(B_{\R}(0,\delta^2)  ;l_{e,c}^{-a})$, we set
$b_{\mathbf m,0}(|z_0|^2   )=b_{\mathbf m,0}  $ and $b_{\bar{\mathbf m},j}=\bar{b}_{\mathbf m,j}$.
Then, by \eqref{25} and \eqref{26}, we see that the Hamiltonian vector field $X_\chi$ is given by
\begin{align*}
&(X_\chi)_k(z,\eta   )=W_k(z,\eta,\rho(z), b(\rho(z)   ), B(\rho(z)   )   )+ Y_j(z,\eta   ),\\&
(X_\chi)_\eta(z,\eta   )=W_\eta(z,\eta,\rho(z), b(\rho(z)   ), B(\rho(z)   ),
\end{align*}
where $\rho(z)=(|z_1|^2,|z_2|^2)$, $b(\rho   )=\{b_{\mathbf m,j}(\rho_j   )\}_{j=0,1,2,\mathbf m\in\mathbf{NR}(M)}$ and\\ $B(\rho   )=\{B_{\mathbf m,j}(\rho_j   )\}_{j=1,2,\mathbf m\in \mathbf{NR}(M-1,j)}$.
\begin{align*}
W_k(z,\eta,\rho,b,B   )=&-2\im (1+\tilde \gamma(\rho_k))\left[\sum_{j=0,1,2} \sum_{\mathbf m\in \mathbf {NR}(M)}\nu_k b_{ \mathbf m, j}\frac{\mathbf Z^{\mathbf m}}{\bar z_k}\right.\\&\quad\quad \left.+\frac 1 2 \sum_{j=1,2} \sum_{\mathbf m\in \mathbf {NR}(j,M-1)}\(\nu_k\frac{\mathbf Z^{\mathbf m}}{\bar z_k} B_{ \mathbf m, j},\eta\)+\(\mu_k \frac{\mathbf Z^{\bar{\mathbf m}}}{\bar z_k} \bar B_{ \mathbf m, j},\bar \eta\)\right],\\
W_\eta(z,\eta,\rho,b,B   )=&-\im \sum_{j=1,2} \sum_{\mathbf m\in \mathbf{NR}(j,M-1)}\mathbf Z^{\mathbf m}B_{ \mathbf m, j},
\end{align*}
with $\rho=(\rho_1,\rho_2)$ and
\begin{align*}
 Y_k(z,\eta   )=&-2\im (1+\tilde \gamma(|z_k|^2))\left[\sum_{\mathbf m\in \mathbf {NR}(M)} b_{ \mathbf m, k}'(|z_k|^2   )z_k\mathbf Z^{\mathbf m}\right.\\&\quad\quad \left.+\frac 1 2 \sum_{\mathbf m\in \mathbf {NR}(k,M-1)}\(z_k\mathbf Z^{\mathbf m}B_{ \mathbf m, k}'(|z_k|^2   ),\eta\)+\(z_k\mathbf Z^{\bar{\mathbf m}} \bar B_{ \mathbf m, k}'(|z_k|^2   ),\bar \eta\)\right],
\end{align*}
Notice that we have 
\begin{align}\label{51}
Y_k \bar z_k+ Y_{\bar k}z_k=0.
\end{align}
Further, we have
\begin{align*}
&\mathbf W_k(z,\eta,\rho,b,B   ):=z_kW_{\bar k}(z,\eta,\rho,b,B   )+\bar z_k W_k(z,\eta,\rho,b,B   ) =\\& 2 (1+\tilde \gamma(\rho_k))(\mu_k-\nu_k)\left[\im\sum_{j=0,1,2} \sum_{\mathbf m \in \mathbf {NR}(M)} b_{ \mathbf m, j}\mathbf Z^{\mathbf m}+\sum_{j=1,2} \sum_{\mathbf m\in \mathbf {NR}(j,M-1)}\< \im \mathbf Z^{\mathbf m} B_{ \mathbf m, j}, \eta\>\right].
\end{align*}

We set $(r_z(z,\eta   )(s),r_\eta(z,\eta   )(s))=(r_z(s),r_\eta(s))=(r_1(s),r_2(s),r_\eta(s))$ to be a solution of 
\begin{align*}
&\frac{d}{ds}(z_k+r_k(s))=(X_{\chi})_k(z+r_z(s),\eta+r_\eta(s)   )\\&\frac{d}{ds}(\eta+r_\eta(s))=(X_{\chi})_\eta(z+r_z(s),\eta+r_\eta(s)   ),
\end{align*}
with $(r_z(0),r_\eta(0))=(0,0)$.
Equivalently, we are setting $(r_z(s),r_\eta(s))$ to be the solution of
\begin{align}
&r_z(z,\eta   )(s)=\int_0^s (X_{\chi})_z(z+r_z(\tau),\eta+r_\eta(\tau)   )\,d\tau,\label{52}\\&
r_\eta(z,\eta   )(s)=\int_0^s (X_{\chi})_\eta(z+r_z(\tau),\eta+r_\eta(\tau)   )\,d\tau.\label{53}
\end{align}
We set
\begin{align*}
(r_z(z,\eta   ),r_\eta(z,\eta   )):=(r_z(z,\eta   )(1),r_\eta(z,\eta   )(1)).
\end{align*}

By standard argument, we have the following lemma.
\begin{lemma}\label{lem:prb1}
Let $\delta>0$ sufficiently small.
Then, there exists $$(r_z,r_\eta)\in C^\omega(B_{\C^2\times l_{e,c}^{-a}}(0,\delta)  ;C([0,1];\C^2\times l_{e,c}^a)),$$ s.t.
$(r_z(z,\eta   )(s),r_\eta(z,\eta+\lambda)(s))$ is the solution of system \eqref{52}--\eqref{53} and
\begin{align}
&|r_{z}(z,\eta   )|+\|r_\eta(z,\eta   )\|_{l_e^{a}}\lesssim |z||z_1z_2|^{M-1}+|z|^2 |z_1z_2|^{M-2}\|\eta\|_{l_e^{-a}},\nonumber\\&
|z+r_z(z,\eta   )|^2-|z|^2\lesssim |z_1z_2|^{M}+|z| |z_1z_2|^{M-1}\|\eta\|_{l_e^{-a}}.\label{55.1}
\end{align}
\end{lemma}

\begin{proof}
We only prove \eqref{55.1}.
\begin{align*}
|z_k+r_k|^2-|z_k|^2=\int_0^1 \frac{d}{ds} |z_k+r_k(z,\eta,\lambda,s)|^2\,ds=\int_0^1 \((z_k+r_k)W_{\bar k}+\overline{(z_k+r_k)}W_k\)\,ds.
\end{align*}
Therefore, we have the conclusion.
\end{proof}

We set $w_z$ and $w_\eta$ to be the solution of the following integral equation.
\begin{align}
&w_z(z,\eta,\rho,b,B   )(s)=\int_0^s W_z(z+w_z(\tau),\eta+w_\eta(\tau),\rho,b,B   )\,d\tau,\label{54}\\
&w_\eta(z,\eta,\rho,b,B   )(s)=\int_0^s W_\eta(z+w_z(\tau),\eta+w_\eta(\tau),\rho,b,B   )\,d\tau.\label{55}
\end{align}
The existence of such $w_z$, $w_\eta$ are standard.
We set $w_z(z,\eta,\rho,b,B   ):=w_z(z,\eta,\rho,b,B   )(1)$ and $w_\eta(z,\eta,\rho,b,B   ):=w_\eta(z,\eta,\rho,b,B   )(1)$

We set
\begin{align*}
X:=\C^A\times (l_{e,c}^a)^B, 
\end{align*}
where
$A=\sharp\{(\mathbf m,j)\in \mathbf{NR}(M)\times \{0,1,2\}\}$ and $B=\sharp\{(\mathbf m, j)\ |\ j=1,2,\ \mathbf m\in \mathbf{NR}(M-1,j)\}$.

The contribution of $w_z$ and $w_\eta$ are given by the following, which can be obtained by mere substitution.
\begin{lemma}\label{lem:prb3}
For $R>0$, there exist $\delta>0$ s.t. there exists $$(w_z,w_\eta)\in C^\omega(B_{\C^2\times l_{e,c}^{-a}}(0,\delta)\times B_{\R^2}(0,\delta^2)\times B_X(0,R)  ;C([0,1];\C^2\times l_{e,c}^a)),$$ s.t.
$(w_z(z,\eta,\rho,b,B   )(s),r_\eta(z,\eta,\rho,b,B   )(s))$ is the solution of system \eqref{54}--\eqref{55} for $$(z,\eta,\rho,b,B   )\in B_{\C^2\times l_{e,c}^{-a}}(0,\delta)\times B_{\R^2}(0,\delta^2)\times B_X(0,R)  $$ and
\begin{align*}
&|w_{z}(z,\eta,\rho,b,B   )|+\|w_\eta(z,\eta,\rho,b,B   )\|_{l_e^{a}}\lesssim |z||z_1z_2|+|z|^2 \|\eta\|_{l_e^{-a}}.
\end{align*}
\end{lemma}

\begin{lemma}\label{lem:prb4}
We have
\begin{align}
&|r_z(z,\eta,\lambda)-w_z(z,\eta,\lambda)|+\|r_\eta(z,\eta,\lambda)-w_\eta(z,\eta,\lambda)\|_{l_e^{a}}\lesssim |z| |z_1z_2|^{M}+|z|^2 |z_1z_2|^{M-1}\|\eta\|_{l_e^{-a}},\label{56}\\
&|z_k+r_k|^2-|z_k+w_k|^2\lesssim |z| |z_1z_2|^{M+1}+|z|^2 |z_1z_2|^M\|\eta\|_{l_e^a}+|z|^3 |z_1z_2|^{M-1}\|\eta\|_{l_e^a}^2.\label{57}
\end{align}
\end{lemma}

\begin{proof}
We have
\begin{align*}
r_k-w_k=&\int_0^1  Y_k(z_z+r_z(s),\eta+r_\eta(s)   )\,ds \\&+ \int_0^1 \(W_k(z_k+r_k(s),\eta+r_\eta(s),\rho(z+r_z)   )-W_k(z+w_z(s),\eta+w_\eta(s),\rho(z)   )\)\,ds.
\end{align*}
The first integral can be bounded by $|z| |z_1z_2|^{M}+|z|^2 |z_1z_2|^{M-1}\|\eta\|_{l_e^{-a}}$.
For the second integral, using Taylor expansion again, the terms with $r_k-w_k$ or $r_\eta-w_\eta$ can be absorbed in the l.h.s.\ of \eqref{56}.
The for the term with $\rho(z+r_z)-\rho(z)$ can be bounded using Lemma \ref{lem:prb1}.
Therefore, we have \eqref{56}.
We skip the proof of \eqref{57}.
\end{proof}

By lemmas \ref{lem:prb3} and \ref{lem:prb4}, we see that the only part which affects the terms with $\mathbf m\in \mathbf{NR}(M)$ or $\mathbf m \in \mathbf{NR}(M-1,j)$ in the expansion of $E^{M-1}(z+r_z,\eta+r_\eta   )$ will be $w_z$ and $w_\eta$.

\begin{lemma}
We have
\begin{align*}
&|z_k+w_k|^2-|z|^2=2 (\mu_k-\nu_k)\left[\im\sum_{j=1,2} \sum_{\mathbf m \in \mathbf {NR}(2)} b_{\mathbf m,j}\mathbf Z^{\mathbf m}+\sum_{j=1,2} \sum_{\mathbf m\in \mathbf {NR}(1,j)}\< \im \mathbf Z^{\mathbf m} B_{ \mathbf m, j}, \eta\>\right]\\&
+\sum_{j=1,2}\sum_{\mathbf m \in \mathbf M(M)}c_{\mathbf m,j}\(|z_j|^2,\{ b_{\mathbf n, j}(|z_j|^2   ) \}_{\mathbf n \in \mathrm{NR}(M)},\{ B_{\mathbf n,j}(|z_j|^2   )\}_{\mathbf n\in \mathbf{NR}(M-1,j)}   \)\mathbf Z^{\mathbf m} \\&+\sum_{j=1,2}\sum_{\mathbf m\in \mathbf M(M-1,j)}\< \mathbf Z^{\mathbf m}g_{\mathbf m,j}\(|z_j|^2,\{ b_{\mathbf n, j}(|z_j|^2   ) \}_{\mathbf n \in \mathrm{NR}(M)},\{ B_{\mathbf n,j}(|z_j|^2   )\}_{\mathbf n\in \mathbf{NR}(M-1,j)}   \),\eta\>\\&+R_k,\\
&w_\eta=-\im \sum_{j=1,2} \sum_{\mathbf m\in \mathbf{NR}(M-1,j)}\mathbf Z^{\mathbf m}B_{ \mathbf m, j}\\&+\sum_{j=1,2}\sum_{\mathbf m \in \mathbf M(M)}d_{\mathbf m,j}\(|z_j|^2,\{ b_{\mathbf n, j}(|z_j|^2   ) \}_{\mathbf n \in \mathrm{NR}(M-1)},\{ B_{\mathbf n,j}(|z_j|^2   )\}_{\mathbf n\in \mathbf{NR}(M-1,j)}   \)\mathbf Z^{\mathbf m} \\&+R_\eta,
\end{align*}
where $c_{\mathbf m,j}(0,b,B   )=d_{\mathbf m,j}(0,b,B   )=0$, $g_{\mathbf m,j}(0,b,B   )=0$,
$|R_k|\lesssim |z_1z_2|^{M+1}+|z_1z_2|\|\eta\|_{l_e^{-a}}+\|\eta\|_{l_e^{-a}}^2$ and $|R_\eta|\lesssim |z_1z_2|^M+|z_1z_2|^{M-1}\|\eta\|_{l_e^{-a}}$.
\end{lemma}

We now prove Proposition \ref{prop:Mth}.

\begin{proof}[Proof of Proposition \ref{prop:Mth}]
We compute $E^{M-1}(z+r_z,\eta+r_\eta   )$.
Notice that by Lemma \ref{lem:prb4}, no effect of $r_z-W_z$, $r_\eta-W_\eta$ and $|z_k+r_k|^2-|z_k|^2-|z+w_k|^2$ appears in terms which we are in concern.

We can write $E^{M-1}(z+r_z,\eta+r_\eta   )$ as
\begin{align}
&E^{M-1}(z+r_z,\eta+r_\eta   )=E_1(|z_1|^2   )+E_2(|z_2|^2   )+E(\eta   )\label{53.1}\\&+\sum_{2\leq k\leq M-1}\sum_{j=1,2}\sum_{\mathbf m \in  \mathbf R(k)}C_{\mathbf m,j}^{M-1}(|z_j|^2   ) \mathbf Z^{\mathbf m}+\sum_{1\leq k\leq M-2}\sum_{j=1,2}\sum_{\mathbf m \in \mathbf R(k,j)}\<\mathbf Z^{\mathbf m}G_{\mathbf m,j}^{M-1}(|z_j|^2   ),  \eta\>\nonumber\\&+\sum_{k\geq M} \sum_{j=0,1,2} \sum_{\mathbf m \in \mathbf M( k)}\tilde C_{\mathbf m,j}^{M-1}(|z_j|^2   )\mathbf Z^{\mathbf m} + \sum_{k\geq M-1}\sum_{j=1,2}\sum_{\mathbf m \in \mathbf M(k,j)}\<\mathbf Z^{\mathbf m}\tilde G_{\mathbf m,j}^{M-1}(|z_j|^2   ),\eta\>\nonumber\\&+ \tilde R_{M-1}(z,\eta   ).\nonumber
\end{align}
Notice that the terms in the first and second line is not affected by $r_z$, $\eta_z$.
This is because of Lemma \ref{lem:prb1}.
We want to have
\begin{align*}
\tilde C^{M-1}_{\mathbf m,j}(|z_j|^2   )=0\ \mathrm{for}\  \mathbf m\in \mathbf {NR}(M),\quad \mathrm{and}\quad \tilde G^{M-1}_{\mathbf m,j}(|z_j|^2   )=0\ \mathrm{for}\  \mathbf m\in \mathbf {NR}(M,j).
\end{align*}
We first compute $\tilde C^{M-1}_{\mathbf m,0}  $.
The only source besides $C^{M-1}_{\mathbf m ,0}  $ are the terms coming from $E_j(|z_j+r_z|^2)$.
So, we have
\begin{align*}
\tilde C^{M-1}_{\mathbf m, 0}  =\im \mathbf e\cdot (\mu-\nu) b_{\mathbf m,0}+C^{M-1}_{\mathbf m,0}.
\end{align*}
Therefore, we set
\begin{align*}
b_{\mathbf m,0}=\frac{\im C^{M-1}_{\mathbf m,0}}{\mathbf e\cdot (\mu-\nu)}.
\end{align*}
Notice that the relation $b_{\bar{\mathbf m},0}=\overline{b_{\mathbf m,0}}$ is satisfied.

We next compute $\tilde C^{M-1}_{\mathbf m,j}$ and $\tilde G^{M-1}_{\mathbf m,j}$ for $j=1,2$.
We will have
\begin{align}
\tilde C^{M-1}_{\mathbf m,j}(|z_j|^2   )=2\im \mathbf e\cdot (\mu-\nu) b_{ \mathbf m, j} + T_{\mathbf m,j}(|z_j|^2,\{b_{\mathbf n,j}\}_{\mathbf n\in \mathbf{NR}(M-1)},\{B_{\mathbf n,j}\}_{\mathbf n\in \mathbf {NR}(M-2,j)}  ),\nonumber\\
\tilde G^{M-1}_{\mathbf m,j}(|z_j|^2  )=-2 \im (H-\mathbf e\cdot (\mu-\nu)) B_{\mathbf m, j} + \mathcal T_{\mathbf m,j}(|z_j|^2,\{b_{\mathbf n,j}\}_{\mathbf n\in \mathbf{NR}(M-1)},\{B_{\mathbf n,j}\}_{\mathbf n\in \mathbf{NR}(M-2,j)} 
 ),\label{nonres}
\end{align}
where $T(0,b_{\mathbf m,j},B_{\mathbf m,j}  )=0$, $\mathcal T(0,b_{\mathbf m,j},B_{\mathbf m,j}  )=0$ and $T$, $\mathcal T$ depends on $b_{\mathbf n,j}$, $B_{\mathbf n,j}$ linearly.
Notice that in \eqref{nonres}, we have $\mathbf e\cdot (\mu-\nu) $ and $H-\mathbf e\cdot(\mu-\nu)$ which are invertible if and only if $\mathbf m\in \mathbf{NR}(M)$ and $\mathbf m\in \mathbf{NR}(M-1,j)$ respectively.
It is obvious now that we can choose $b_{\mathbf m,j}$, $B_{\mathbf m,j}$ by implicit function theorem.
\end{proof}

\subsection{Proof of the decay estimate Lemma \ref{lem:l4}}\label{sec:proofl4}

In this section, we prove Lemma \ref{lem:l4}.
Set $\phi_1,\phi_2,\phi_3\in C_0^\infty(\R;\R)$ s.t.
\begin{align*}
&\phi_1(x)+\phi_2(x)+\phi_3(x)=1,\ x\in[0,4],\\&
\mathrm{supp}\phi_1\subset (-1,\omega_*/2),\ \mathrm{supp}\phi_2\subset (\omega_*/4,\omega_*+\frac3 4 (4-\omega_*)),\ 
\mathrm{supp}\phi_3\subset (\omega_*+\frac 1 2 (4-\omega_*),5),\ 
\end{align*}
To prove Lemma \ref{lem:l4}, it suffices to show
\begin{align}\label{60}
\|e^{-\im t H} R_H^+(\omega_*)\phi_j(H)P_c \|_{l^{2,\sigma}\to l^{2,-\sigma}} \lesssim t^{-3/2},\quad j=1,2,3.
\end{align}
for $t>1$.
The estimate for $j=1$ and $j=3$ is similar so we only show it for $j=1,2$.

Before proving the estimate for $j=1$, we prepare an elementary lemma.
\begin{lemma}\label{lem:10}
Let $g\in C_0^\infty(\R;\R)$.
Then, $\|g(H)\|_{l^{2,-3}\to l^{2,-3}}\lesssim 1$.
\end{lemma}

The following local decay estimate was given by Pelinovsky-Stefanov \cite{PS08JMP},
\begin{lemma}\label{lem:11}
Let $\sigma>7/2$.
Then
\begin{align*}
\|P_c e^{-\im t H}\|_{l^{2,\sigma}\to l^{2,-\sigma}}\lesssim |t|^{-3/2}.
\end{align*}
\end{lemma}

We now prove \eqref{60} for $j=1$.
Set $g(x)=\lim_{\delta\downarrow 0}\phi_1(x)(x-\omega_*-\im \delta)^{-1}=\phi_1(x)(x-\omega_*)^{-1}\in C_0^\infty$.
Then,
\begin{align*}
\|e^{-\im t H} R_H^+(\omega_*)\phi_j(H)P_c \|_{l^{2,3}\to l^{2,-3}}&=\|g(H)e^{-\im t H}P_c\|_{l^{2,3}\to l^{2,-3}}\\&\leq \|g(H)\|_{l^{2,-3}\to l^{2,-3}}\|P_c e^{-\im t H}\|_{l^{2,3}\to l^{2,-3}}\lesssim |t|^{-3/2}.
\end{align*}
Therefore, we get the estimate for $j=1$.

Next we show the estimate for $j=2$.
First, notice that
\begin{align*}
e^{-\im t H}R_H^+(\omega_*)=\lim_{\varepsilon\downarrow 0}e^{-\im t H}(H-\lambda-\im \varepsilon)^{-1}=\im e^{-\im \lambda}\lim_{\varepsilon\downarrow 0}\int_t^\infty e^{-\im (H-\lambda-\im \varepsilon)s}\,ds.
\end{align*}
Therefore, it suffices to show
\begin{align}\label{62}
\|e^{-\im t H}\phi_2(H)P_c\|_{l^{2,\sigma}\to l^{2,-\sigma}}\lesssim t^{-5/2}.
\end{align}
Indeed,
\begin{align*}
\|e^{-\im t H} R_H^+(\omega_*)\phi_2(H)P_c \|_{l^{2,\sigma}\to l^{2,-\sigma}}&\leq\int_t^\infty \|e^{-\im (H-\lambda)s}\phi_2(H)P_c\|_{l^{2,\sigma}\to l^{2,-\sigma}}\,ds\\&\lesssim \int_t^\infty s^{-5/2}\,ds\lesssim t^{-3/2}.
\end{align*}

To prove lemma \ref{62}, we show the following lemma.
\begin{lemma}\label{lem:12}
Let $\sigma>7/2$. Then we have
\begin{align*}
\left\|\frac{d^3}{d \omega^3}R_{H}^+(\omega)\right\|_{l^{2,\sigma}\to l^{2,-\sigma}}\lesssim_K 1,
\end{align*}
for a compact $K\subset [0,4]$.
\end{lemma}

\begin{proof}
See Corollary 6.1 of \cite{KKK06AA}.
\end{proof}

By Lemma \ref{lem:12}, we immediately have \eqref{62}.
Indeed,
\begin{align*}
e^{-\im t H}\phi_2(H)=\frac{1}{2\pi \im} \int_0^4 e^{-\im t \omega}\phi_2(\omega) \Im R(\omega)\,d \omega
\end{align*}
Therefore, by integrating by parts, we have
\begin{align*}
\|e^{-\im t H}\phi_2(H)P_c\|_{l^{2,\sigma}\to l^{2,-\sigma}}\lesssim t^{-3}\int_0^4 \|\frac{d^3}{d \omega^3}R_H^+(\omega)\|_{l^{2,\sigma}\to l^{2,-\sigma}}\,d \omega \lesssim t^{-3}.
\end{align*}
Therefore, we have the conclusion.

\appendix

\section{Proof of the formula \eqref{FGRexpression}}
Let $\hat f$ be distorted Fourier transform related to $H P_c$ (see \cite{Cuccagna09JMAA}).
We want to compute the constant $\Gamma=\Im (R_H^+(\omega_*)G,G)$ appearing in the assumption \eqref{FGR}, where $\omega_*\in (0,4)$.
First, recall
\begin{align*}
(A-\im \varepsilon)^{-1}=\frac{A^2}{A^2+\varepsilon^2}\frac{1}{A}+\im \frac{\varepsilon}{A^2+\varepsilon^2}.
\end{align*}
Therefore,
\begin{align*}
\Im ((H-\omega_*-\im \varepsilon)^{-1}G, G)=\varepsilon\int_{\T}\frac{1}{(2-2\cos \xi-\omega_*)^2+\varepsilon^2}|\hat G(\xi)|^2\,d\xi=\frac{\varepsilon}{4}\int_{\T}\frac{|\hat G(\xi)|^2d\xi}{(\tilde \omega-\cos \xi)^2+\varepsilon^2},
\end{align*}
where $\tilde \omega=\frac{1}{2}(2-\omega_*)\in (-1,1)$.
Further, 
\begin{align*}
\frac{\varepsilon}{4}\int_{\T}\frac{d\xi}{(\tilde \omega-\cos \xi)^2+\varepsilon^2}=-\frac{\im}{8}\int_\T\(\frac{1}{(\tilde \omega-\cos \xi-\im \varepsilon)}-\frac{1}{(\tilde \omega-\cos \xi+\im \varepsilon)}\)\,d\xi
\end{align*}
Now, set $\xi_{\varepsilon,\pm}=\cos^{-1}(\tilde \omega\mp \im \varepsilon)$.
We have $\Re \xi_{\varepsilon,+}=-\Re \xi_{\varepsilon,-}\in (0,\pi)$ and $\Im \xi_{\varepsilon,+}=\Im \xi_{\varepsilon,-}$.
$\xi_{\varepsilon,\pm}\to \pm \arccos \tilde \omega$ ($\arccos$ is the inverse of $\cos|_{[0,\pi]}$).
Thus, since the residue of $(\tilde \omega-\cos \xi \mp \im \varepsilon)^{-1}$ at $\xi=\xi_{\varepsilon,\pm}$ is $\frac{1}{\sin \xi_{\varepsilon,\pm}}$
\begin{align*}
\frac{\varepsilon}{4}\int_{\T}\frac{d\xi}{(\tilde \omega-\cos \xi)^2+\varepsilon^2}=\frac{\pi}{4}\mathrm{Res}_{\xi=\xi_{\varepsilon,\pm}}(\tilde \omega -\cos \xi\pm \im \varepsilon)^{-1}=\frac{\pi}{4}\(\frac{1}{\sin \xi_{\varepsilon,+}}-\frac{1}{\sin \xi_{\varepsilon,-}}\) \to \frac{\pi}{2 \sin \xi_+}.
\end{align*}
Thus, we have
\begin{align*}
\Im ((H-\omega_*-\im \varepsilon)^{-1}G, G)\to \frac{\pi}{4\sin (\arccos \tilde \omega)} \sum_{\pm}|\hat G(\pm \arccos \tilde \omega)|^2.
\end{align*}
As a conclusion, we have
\begin{align*}
\Im (R_H^+(\omega_*)G,G)=\frac{\pi}{4\sin (\arccos (\frac{1}{2}(2-\omega_*)))} \sum_{\pm}|\hat G(\pm \arccos (\frac{1}{2}(2-\omega_*)))|^2.
\end{align*}

\subsection*{Acknowledgments}   
The author was supported by the JSPS KAKENHI Grant numbers JP15K17568, JP17H02851 and JP17H02853.
The author thank valuable suggestions from Scipio Cuccagna and Kenji Nakanishi.
Further, he is grateful for helpful comments given by the anonymous referees to improve the presentation of the paper.

Department of Mathematics and Informatics,
Faculty of Science,
Chiba University,
Chiba 263-8522, Japan

{\it E-mail Address}: {\tt maeda@math.s.chiba-u.ac.jp}

\end{document}